\documentclass{article}
\usepackage{spconf,amsmath,graphicx,hyperref, amsthm}
\usepackage{algorithmic}
\usepackage{array}
\usepackage{IEEEtrantools}
\usepackage{enumitem}
\usepackage{subcaption}
\usepackage{textcomp}
\usepackage{stfloats}
\usepackage{url}
\usepackage{verbatim}
\usepackage{float}
\usepackage{graphicx}
\usepackage{cite}
\usepackage{mathtools}
\usepackage[T1]{fontenc}
\usepackage{aecompl}
\usepackage{amsmath,amsthm,amssymb,scrextend}
\usepackage{bm}
\usepackage{fancyhdr}
\usepackage[labeled]{multibib}
\newcites{supp}{Appendix References}

\usepackage[ruled,linesnumbered]{algorithm2e}

\usepackage{tikz}

\usepackage{pgfplots}
\pgfplotsset{compat=1.18}
\usepackage[mathscr]{euscript}
 \let\mathscr\relax
\usepackage[scr]{rsfso}

\usepackage{multicol}
\usepackage{multirow}
\usepackage{subfiles}

\newcommand{\fnorm}{\Vert_{\mathrm{F}}}

\newcommand{\gi}{G_{i}}
\newcommand{\gj}{G_{j}}

\DeclareMathOperator{\vect}{vec}
\DeclareMathOperator{\BlkDiag}{BlkDiag}
\newcommand{\SEsync}{$\mathrm{SE}$-Sync\xspace}

\newcommand{\tr}{\text{Tr}}
\newcommand{\od}{\mathrm{O}(d)}
\newcommand{\sod}{\mathrm{SO}(d)}
\newcommand{\SEd}{\mathrm{SE}(d)}
\newcommand{\Mt}{M_t}

\newcommand{\identity}{I_d}

\theoremstyle{remark}

\newtheorem{lemma}{Lemma}

\newtheorem{Theorem}{Theorem}

\newtheorem{Proposition}{Proposition}

\numberwithin{Theorem}{section}
\numberwithin{Definition}{section}
\numberwithin{lemma}{section}
\numberwithin{Algorithm}{section}
\numberwithin{equation}{section}
\numberwithin{Corollary}{section}
\numberwithin{problem}{section}

\title{Anchored Spectral Estimator for Rigid Motion Synchronization}
\name{Ziyue Zhao$^{\dagger}$\qquad Huikang Liu$^{\star}$\qquad Man-Chung Yue$^{\mathsection}$}
\address{$^{\dagger}$Department of Data and Systems Engineering, The University of Hong Kong\\$^{\star}$Antai College of Economics and Management, Shanghai Jiao Tong University\\$^{\mathsection}$Department of Applied Mathematics, The Hong Kong Polytechnic University\\
\url{zyuezhao@connect.hku.hk}, \url{hkl1u@sjtu.edu.cn}, \url{manchung.yue@polyu.edu.hk}
}
\begin{document}
\ninept
\maketitle
\begin{abstract}
A rigid motion in $\mathbb{R}^d$ consists of a proper rotation and a translation, and it can be represented as a matrix in $\mathbb{R}^{(d+1)\times (d+1)}$. The problem of rigid motion synchronization aims to estimate a collection of rigid motions $G^*_1, \dots, G^*_n$ from noisy observations of their comparisons ${G^*_i}^{-1} G^*_j$. Such problems naturally arise in diverse applications across signal processing, robotics, and computer vision, and have thus attracted intense research attention in recent years. Motivated by geometric considerations, this paper develops a novel spectral approach for rigid motion synchronization, called the anchored spectral estimator (ASE). Theoretically, we establish uniform estimation error bounds for the estimators produced by ASE. Empirically, we show that ASE outperforms the widely used two-stage approach, which first estimates the rotations and then the translations. Further numerical experiments on the multiple point-set registration problem are presented to demonstrate the superiority of ASE over state-of-the-art methods.
\end{abstract}
\begin{keywords}
Special Euclidean group synchronization, Rigid motions, Spectral method, Estimation error bound
\end{keywords}

\section{Introduction}
\label{sec:intro}

A rigid motion in $\mathbb{R}^d$ consists of a proper rotation and a translation and can be represented as a $(d+1)\times (d+1) $ matrix. The problem of rigid motion synchronization is to recover a collection of rigid motions $G_1^*,\dots, G_n^*\in \SEd$ from the noisy observations of their pairwise comparisons $ G_i^{*-1} G_j^{*}$. Since the set of rigid motions forms the special Euclidean group~$\SEd$, the problem is also called special Euclidean group synchronization. We denote it by \SEsync for simplicity. \SEsync finds a wide range of applications across many scientific and engineering domains, including simultaneous localization and mapping~\cite{briales2017cartan, chen2025non}, structure from motion~\cite{eriksson2019rotation, ozyecsil2017survey}, multiple point-set registration~\cite{yew2021learning}, sensor network localization~\cite{so2007theory}, and cryogenic electron microscopy~\cite{singer2011three}. 

\SEsync is very challenging from a computational perspective, hindering its real-world applications. Indeed, it contains NP-hard problems as special cases~\cite{liu2017estimation}. Moreover, the natural approach via its least square estimator results in a non-convex and nonlinear constrained optimization problem. Existing approaches to \SEsync and more generally other synchronization problems can roughly be divided into three categories: semidefinite relaxations (SDR), non-convex algorithms, and spectral estimators. The SDR approach relaxes the non-convex least square estimation problem to a convex semidefinite programming problem by lifting techniques~\cite{bandeira2017tightness, carlone2015lagrangian, rosen2019se, ling2023solving}. Many SDRs for synchronization problems are proved to be tight ({\it i.e.}, the optimal solutions to the relaxed problem are also
optimal to the original non-convex problem). However, the resulting semidefinite programming problem is of dimension $O(nd) \times O(nd)$ and hence prohibitively expensive to solve. Non-convex algorithms directly solve the non-convex optimization problem associated with the least squares estimator by iterative local search algorithms~\cite{boumal2016nonconvex, liu2017estimation, ling2022improved, gao2022iterative, liu2023unified, gao2023optimal}. These algorithms often enjoy strong statistical guarantees if they are suitably initialized~\cite{boumal2016nonconvex, liu2017estimation, ling2022improved, gao2022iterative, liu2023unified, gao2023optimal}. Finally, spectral estimators first compute a small number of eigenvectors of a data matrix constructed from the comparison observations and then generate the output by applying a certain rounding procedure to the eigenvectors~\cite{singer2011angular, cucuringu2012sensor, arrigoni2020synchronization, arrigoni2016spectral, hadi2024se, doherty2022performance}. Although spectral estimators are used as an initialization in many non-convex algorithms, they often enjoy an estimation error that is of the same order as those non-convex algorithms~\cite{zhang2024exact}. Furthermore, they are highly efficient and easy to implement. In view of these advantages, this paper focuses on spectral estimators.

In the context of \SEsync, there have been a number of existing spectral estimators, and they can be subdivided into two types. The first type includes~\cite{hartley2013rotation, govindu2001combining} and adopts a divide-and-conquer, two-stage strategy that first estimates the rotations using only the rotation comparisons extracted from the full observations, but ignores the translation comparisons. The translations are then estimated by a separate procedure based on the rotations estimated in the first stage. In contrast, the second type jointly estimates these two constituents of the rigid motions. Examples of the second type include~\cite{arrigoni2016spectral, doherty2022performance, hadi2024se}. Since the translation comparisons are coupled with the rotations and hence can provide extra information for rotation estimation, the second type generally performs better than the first when the noise on the translation comparisons is not too large.

On the other hand, regardless of the type, the eigenvectors obtained from the first step of spectral estimators suffer from a symmetry issue; see Section~\ref{sec:ase} for details. If the issue is not properly handled, the performance of the spectral estimators can be significantly deteriorated. 

\subsection{Contributions}
In this paper, we propose and analyze a new spectral estimator for \SEsync. Our contributions are as follows.

\begin{itemize}
    \item First, exploiting the special geometry of the set of rotations, we develop a novel spectral estimator called the anchored spectral estimator (ASE), which uses anchored projections for the rounding step to tackle the symmetry issue.

    \item Second, we establish a strong theoretical guarantee for ASE, which asserts that its estimators enjoy a uniform estimation error bound of the order $O((\sigma_1 + \sigma_2^2)(\sqrt{d} + \sqrt{\log n})d/\sqrt{n})$, where $\sigma_1^2$ and $\sigma_2^2$ are the variances of the noise on the rotation and translation comparisons, respectively.

    \item Third, we empirically verify the advantage of ASE over a two-stage approach for solving \SEsync through synthetic data. We also demonstrate the superiority of ASE over state-of-the-art spectral estimators of the second type via a numerical experiment on the multiple point-set registration problem using real datasets.
\end{itemize}

\subsection{Notation}
For a vector $x$, $\|x\|_2$ denotes its 2-norm. 
For a matrix $X$, $X^\top$, $\|X\|$ and $\|X\|_{\mathrm{F}}$ denote its transpose, operator norm, and Frobenius norm, respectively. Also, $\vect(X)$ denotes the vector obtained by stacking its columns. Given a sequence $\{x_i\}_{i=1}^n\in \mathbb{R}^{a\times b}$ of vectors or matrices, $\BlkDiag (x_1,\dots, x_n) \in \mathbb{R}^{na\times nb}$ denotes the block matrix with its $i$-th diagonal block being $x_i$. Given two matrices $X$ and $Y$, their Kronecker product is denoted by $X\otimes Y$.
The $a\times a$ identity matrix and the all-one matrix are denoted by $I_a$ and $J_a$.
The sets of $a\times a$ orthogonal and special orthogonal matrices are denoted by $\mathrm{O}(a) = \{ Q\in \mathbb{R}^{a \times a}\mid Q Q^\top=I_a
\}$ and $\mathrm{SO}(a) = \{ Q\in \mathbb{R}^{a\times a}\mid QQ^\top=I_a,\ \operatorname{det}(Q)=1
\}$, respectively. We also denote by $[n]$ the set $\{1,\dots,n\}$.

\section{A New Spectral Method for \SEsync}\label{sec:gsp}

\subsection{\SEsync and Its Least Squares Estimator}
This section formally introduces the special Euclidean group synchronization problem (\SEsync) and our proposed spectral method ASE. To begin, recall that a rigid motion in $\mathbb{R}^d$ consists of a proper rotation $R\in \mathrm{SO}(d)$ and a translation $t\in\mathbb{R}^d$, and can be represented as a $(d+1)\times (d+1)$ matrix of the form
\begin{equation*}
    \begin{pmatrix}
        R^\top & t\\
    0 & 1
    \end{pmatrix}.
\end{equation*}
The set of all such matrices is called the special Euclidean group and denoted by $\mathrm{SE}(d)$. In \SEsync, we are interested in estimating a collection of rigid motions $G^*_1, \dots, G^*_n \in\SEd$ based on noisy observations of their comparisons $\gi^{*-1} \gj^*$. More precisely, we assume that we are given the following observations:
\begin{equation*}
    C_{ij} = \gi^{*-1} \gj^*+ W_{ij}, \quad i,j\in [n].
\end{equation*}
In this observation model, the matrix $W_{ij} \in \mathbb{R}^{(d+1)\times (d+1)}$ represents the noise and takes the form
\begin{equation*}
    W_{ij} = 
    \begin{pmatrix}
        W^R_{ij} & w^t_{ij}\\
        0 & 0
    \end{pmatrix},\quad  i,j \in [n],\  i \neq j,
\end{equation*}
where for any $i\neq j$, $W^R_{ij}\in \mathbb{R}^{d\times d}$ has i.i.d. Gaussian entries with mean $0$ and variance $\sigma_1^2$, $w^t_{ij} \in \mathbb{R}^d$ has i.i.d. Gaussian entries with mean $0$ and variance $\sigma_2^2$, and they are statistically independent; and for any $i\in [n]$, $W_{ii}=0$. For any $i\in [n]$, the rotation and translation associated with the ground truth $G^*_i$ are denoted by $R^*_i$ and $t^*_i$, respectively. Without loss of generality, we assume that $\sum_{i = 1}^n t^*_i = 0$.

The least squares estimator, which also coincides with the maximum likelihood estimator under our noise assumption, is defined as any optimal solution to the problem
\begin{equation}\label{opt01}
\begin{split}
    &\min_{G_i \in \SEd}\sum_{i=1}^n\sum_{j=1}^n \| G_i^{-1}G_j - C_{ij}\|_{\mathrm{F}}^2 =\\
    &\min_{\substack{R_i \in \sod, \\ t_i \in \mathbb{R}^d}} 
\sum_{i=1}^n \sum_{j=1}^n \left( 
\| R_i R_j^\top - S_{ij} \|_{\mathrm{F}}^2 + \| t_j - t_i - R^\top_i s_{ij} \|_2^2 
\right),
\end{split}
\end{equation}
where for all $i \neq j$, $S_{ij}=R_i^{*}R_j^{*\top}+W_{ij}^R \in \mathbb{R}^{d\times d}$ and $s_{ij}=s_{ij}^* + w_{ij}^t \in \mathbb{R}^d$ with $s_{ij}^*=R_i^{*}(t_j^*-t_i^*) \in \mathbb{R}^d$. 

Throughout the paper, we adopt the following convention to denote block matrices.
For any sequence $\{x_i\}_{i=1}^n\subseteq \mathbb{R}^{a\times b}$ and two-dimensional sequence $\{y_{ij}\}_{i,j = 1}^n \in \mathbb{R}^{a\times b}$ of vectors or matrices, we denote by $x\in \mathbb{R}^{na\times b}$ and $y\in\mathbb{R}^{na\times nb}$ the block matrices whose $i$-th and $ij$-th blocks are $x_i$ and $y_{ij}$, respectively. For example, $s\in \mathbb{R}^{nd\times n}$ is a block matrix whose $ij$-th block is $s_{ij} \in \mathbb{R}^d$.

We then simplify the least squares estimator~\eqref{opt01}. Let $L=nI_n-J_n$ be the Laplacian matrix of a complete graph, $S \in \mathbb{R}^{nd\times nd}$ be the block matrix whose $ij$-th block is $S_{ij}$, $\hat{T} \in \mathbb{R}^{nd\times n}$ be the block matrix $\hat{T} = \BlkDiag(\sum_{j=1}^n s_{1j}, \dots, \sum_{j=1}^n s_{nj})-s$, and $\hat{\Sigma} \in \mathbb{R}^{nd\times nd}$ be the block matrix $ \hat{\Sigma} = \BlkDiag(\sum_{j =1}^ns_{1j} s_{1j}^\top, \dots,\allowbreak \sum_{j =1}^n s_{nj} s_{nj}^\top)$. Then, we can equivalently reformulate problem~\eqref{opt01} as 
\begin{IEEEeqnarray}{rl}\label{opt02}
\min_{\substack{R \in \sod^n \\ t \in \mathbb{R}^{nd}}} \tr & \Big( 
2R^\top(nI_{nd} - S)R  +  \hat{\Sigma} \nonumber\\
& +  2t^{\top} (L \otimes I_d)t + 2 \mathrm{vec}(R^\top \hat{T})^\top t \Big).
\end{IEEEeqnarray} 
For any fixed $R$, problem~\eqref{opt02} is an unconstrained quadratic optimization and can be solved analytically with optimal solution being $t = -\frac{1}{2n} \vect(R^\top \hat{T} ) $; see~\cite[Lemma~4]{rosen2019se}. We can then eliminate the variable $t$ and further reformulate problem~\eqref{opt01} as
\begin{equation}\label{opt1}
\min_{R \in \sod^n}\tr ( R^\top\Omega R ),
\end{equation}
where 
\begin{equation}\label{eq:data matrix}
\Omega = 2nI_{nd}- 2S + \hat{\Sigma}-\frac{1}{2n}\hat{T}\hat{T}^\top.
\end{equation}

Since $\Omega$ is not positive semidefinite, the objective function $\tr ( R^\top\Omega R )$ is non-convex. Moreover, the block-wise $\mathrm{SO}(d)$ constraint is also non-convex and highly nonlinear. Problem~\eqref{opt1} is therefore difficult to solve.

\subsection{Anchored Spectral Estimator}\label{sec:ase}

We propose a new spectral estimator, called the anchored spectral estimator (ASE), for solving problem~\eqref{opt1} and thus \SEsync; see Algorithm~\ref{alg:algorithm for sed}. Roughly speaking, ASE first relaxes the block-wise special-orthogonality constraint $\sod^n$ to the orthogonality constraint $\Phi^\top \Phi = nI_d$, turning our task into an eigenvalue problem. After computing the $d$ eigenvectors $\Phi\in \mathbb{R}^{nd\times d}$ associated with the smallest eigenvalues of $\Omega$, we round each block $\Phi_i \in\mathbb{R}^{d\times d}$ by projecting a right-rotated version $\Phi_i \Phi_1^\top$ back onto $\mathrm{SO}(d)$ using the projection operator $\Pi_{\mathrm{SO}(d)}$. This yields the estimators $\hat{R}_i$ to the target rotations $R^*_i$ for $i\in [n]$. The projection $\Pi_{\mathrm{SO}(d)}$ can be computed efficiently via a singular value decomposition; see~\cite[Section~5.1.2]{liu2023unified}. Next, the translations are recovered as the optimal solutions to problem~\eqref{opt02} with a fixed $R = \hat{R}$. In other words, $\hat{t}_i = -(\frac{1}{2n}(\hat{T}^\top\otimes \identity)\vect(R^\top))_i = -\frac{1}{2n} ( \vect(\hat{R}^\top \hat{T} ))_i$ for $i\in[n]$. 

\begin{algorithm}[t!]
\caption{Anchored Spectral Estimator for $\SEd$}
\label{alg:algorithm for sed}
\KwIn{The data matrix $\Omega$ defined in \eqref{eq:data matrix}.}\KwOut{Rigid motions $ (\hat{R}_i, \hat{t}_i)$ for $i\in [n]$.}
Compute the eigenvectors $\Phi \in \mathbb{R}^{nd\times d}$ associated with the $d$ smallest eigenvalues of $\Omega$\;
For $i \in [n]$, compute
\[\hat{R}_i= \Pi_{\sod}(\Phi_i\Phi_1^\top) \quad\text{and}\quad \hat{t}_i = -\frac{1}{2n} ( \vect(\hat{R}^\top \hat{T} ))_i .\]
\end{algorithm}

A distinctive feature of ASE is that each block $\Phi_i$ is right-multiplied by $\Phi_1^\top$ before projecting onto $\mathrm{SO}(d)$. In the noiseless case, the block column matrix $\Phi$ consists of the true rotations up to right-multiplication by an unknown orthogonal matrix. The block $\Phi_1$ therefore acts as an anchor, eliminating this global orthogonal ambiguity and correctly locating the blocks $\Phi_i$ within $\mathrm{SO}(d)$. In the presence of noise, Lemma~\ref{le: rotation error bound} in Section~\ref{analysis of error bound} implies that $\Phi_1$ remains close to $R_1^\star$ (up to an orthogonal right-multiplication), and therefore the anchoring mechanism continues to be effective.

\subsection{Other Spectral Estimators}

The spectral method that is closest to ASE is the one developed in~\cite{doherty2022performance}, which uses the same matrix $\Omega$ for the eigenvalue problem in the first step. Nevertheless, unlike our anchored projection, their rounding procedure directly projects each block $\Phi_i$ onto $\sod$. 
Since the blocks $\Phi_i$ may have determinant $-1$, such a rounding procedure suffers from the risk of distorting their relative positions.

There have also been other spectral methods developed for \SEsync. The spectral method in \cite{arrigoni2016spectral} first computes the eigenvectors $\Phi'\in \mathbb{R}^{n(d+1)\times (d+1)}$ of a certain $n(d+1)\times n(d+1)$ data matrix and then projects the blocks $\Phi'_i$ onto $\SEd$. As another example, a spectral method for \SEsync is developed in~\cite{hadi2024se} based on the dual quaternion representation of rigid motions. The method seems to be specific to $d = 3$, and it is unclear how it can be generalized to other dimensions.

We have conducted an experiment on the multiple point-set registration problem to compare our ASE with the spectral methods in \cite{arrigoni2016spectral}, \cite{hadi2024se}, and \cite{doherty2022performance}. The experiment results show that ASE performs the best in terms of estimation error.

\section{Estimation Error Bound}\label{analysis of error bound}

The proposed ASE enjoys a strong theoretical guarantee on its estimation error.
Specifically, we show that the (unsquared) $\ell_2$ estimation error of the estimators $\hat{G}_1,\dots, \hat{G}_n$ enjoys a uniform bound of the order $O((\sigma_1 + \sigma_2^2)(\sqrt{d} + \sqrt{\log n})d/\sqrt{n})$. To state the result, we let $\Mt = \max_{i\in[n]}\| t_i^* \|_2 $.
\begin{Theorem}\label{thm: Theorem of Spectral method in Gaussian noise model} 
Consider Algorithm~\ref{alg:algorithm for sed}.
There exist absolute constants $c_0, c_1, c_2, c_3, c_4, c_5 >0 $ such that if $c_1 \sigma_1 + c_2 \Mt \sigma_2 + c_3 \sigma_2^2 \le \frac{c_0\sqrt{n}}{\sqrt{d}+\sqrt{\log(n)}}$, then
\begin{equation*}
    \begin{split}
        &\max_{i\in [n]}\min_{Q\in \SEd}~\| \hat{G_i} -Q\gi^* \|_{\mathrm{F}} \\
     \le & c_4 \Mt(\Mt^2+1)^2(c_1\sigma_1+c_2\Mt \sigma_2+c_3\sigma_2^2)\frac{(d\sqrt{d}+d\sqrt{\log n})}{\sqrt{n}},
    \end{split}
\end{equation*}
with probability at least $1- c_5 n^{-1}$.
\end{Theorem}

\subsection{Proof Outline}
We next outline the proof of Theorem~\ref{thm: Theorem of Spectral method in Gaussian noise model}, which mainly consists of the three lemmas below. The first one decomposes the (uniform) estimation error of estimated rigid motions into two parts: a (scaled) rotation error and a translation error. This lemma is proved based on \cite[Lemma~2]{liu2023unified}.

\begin{lemma}\label{lem:step 1}
Consider Algorithm~\ref{alg:algorithm for sed}. There exists an absolute constant $c>0$ such that 

\begin{equation*}
    \begin{split}
        & \max_{ i\in[n]} \min_{Q \in \SEd}~\| \hat{G_i} -Q\gi^*\|_{\mathrm{F}} \\
        \le & c \max_{i\in [n]} \| \Phi_i \|\underbrace{ \max_{i\in [n]}\| \Phi_i - R_i^{*} \Bar{Q} \|_{\mathrm{F}}}_{\text{rotation error}}+\underbrace{\max_{i\in [n]} \| \hat{t}_i -(\Pi_{\sod}(\Bar{Q}\Phi_1^\top))^\top t^*_i \|_2 }_{\text{translation error}},
    \end{split}
\end{equation*}
where $\Bar{Q} = \arg \min _{Z \in \od}\allowbreak \| \Phi -R^*Z \|_{\mathrm{F}} $.
\end{lemma}

The second lemma establishes a bound on the (scaled) rotation error, which as discussed in Section~\ref{sec:ase}, provides the theoretical underpinning for the anchoring mechanism of ASE.

\begin{lemma}\label{le: rotation error bound}
Consider Algorithm~\ref{alg:algorithm for sed}. There exist absolute constants $c_0, c_1, c_2, c_3, c_4, c_5 >0 $ such that if $c_1 \sigma_1 + c_2 \Mt \sigma_2 + c_3 \sigma_2^2 \le \frac{c_0\sqrt{n}}{\sqrt{d}+\sqrt{\log(n)}}$, then
\begin{equation*}
    \begin{split}
        &\max_{i\in [n]} \| \Phi_i \| \max_{i\in [n]}\| \Phi_i -R^*_i \Bar{Q} \|_{\mathrm{F}}  \\
        & \le c_4 (\Mt^2+1)^2 (c_1\sigma_1 + c_2 \Mt \sigma_2 + c_3\sigma_2^2) \frac{(d\sqrt{d}+d\sqrt{\log n})}{\sqrt{n}},
    \end{split}
\end{equation*}
with probability at least $1- c_5 n^{-1}$, where $\Bar{Q}$ is defined in Lemma~\ref{lem:step 1}.
\end{lemma}

Lemma~\ref{le: rotation error bound} is derived via the popular leave-one-out technique. To understand the technique, we note that the data matrix $\Omega$ in \eqref{eq:data matrix} can be decomposed as the sum of an informative part and a noise matrix $\Delta$ and that the proof of Lemma~\ref{le: rotation error bound} involves bounding a certain product term between the eigenvectors $\Phi$ and the $i$-th block row $\Delta_i$ of the noise matrix $\Delta$, which are statistically dependent. This obstructs controlling the product term by using concentration inequalities, which rely on statistical independence. To tackle this issue, the leave-one-out technique approximates $\Phi$ by the eigenvectors $\Phi^{(i)}$ of the matrix obtained from deleting the $i$-th block row and $i$-th block column from the matrix $\Delta$. The upshot is that the approximate eigenvectors $\Phi^{(i)}$ are statistically independent of $\Delta_i$, and concentration inequalities can therefore be applied. We emphasize that although the leave-one-out technique has been applied to study other group synchronization problems~\cite{ling2022near,liu2023resync}, the analysis for \SEsync is different and requires additional techniques due to the more complicated structure of the noise matrix $\Delta$.

The final lemma provides a bound on the translation error, the proof of which is similarly based on the leave-one-out technique.

\begin{lemma}\label{le: translation error bound}
Consider Algorithm~\ref{alg:algorithm for sed}. There exist absolute constants $c_0, c_1, c_2, c_3, c_4, c_5 >0 $ such that if $c_1 \sigma_1 + c_2 \Mt \sigma_2 + c_3 \sigma_2^2 \le \frac{c_0\sqrt{n}}{\sqrt{d}+\sqrt{\log(n)}}$, then
\begin{equation*}
    \begin{split}
        &\max_{i\in[n]}\| \hat{t}_i - (\Pi_{\sod}(\Bar{Q}\Phi_1^\top))^\top t_i^* \|_2 \\ 
        & \le   c_4 \Mt(\Mt^2+1)^2( c_1 \sigma_1 + c_2\sigma_2 \Mt + c_3\sigma_2^2)\frac{(d\sqrt{d}+d\sqrt{\log n})}{\sqrt{n}},
    \end{split}
\end{equation*}
with probability at least $1-c_5 n^{-1}$, where $\Bar{Q}$ is defined in Lemma~\ref{lem:step 1}.
\end{lemma}

Theorem~\ref{thm: Theorem of Spectral method in Gaussian noise model} then follows by combining the Lemmas~\ref{lem:step 1}, \ref{le: rotation error bound}, and~\ref{le: translation error bound}.

\section{Numerical Experiments}\label{se:experiments}

In this section, we conduct numerical experiments to investigate the practical performance of ASE. Our codes are available at \url{https://github.com/ziyuezhao1/GSP-SE_D}.

\subsection{Comparison with the Two-Stage Approach}
We first empirically study how the estimation error of ASE varies against the noise magnitudes $\sigma_1$ and $\sigma_2$, especially in comparison with the two-stage approach (see Section~1). To this end, we consider $d=3$ and $n=500$. The ground truth is generated as follows. The rotations $R^*_i$ are generated independently according to the uniform distribution on $\mathrm{SO}(3)$. And the translations $t^*_i$ are generated independently according to the 3-dimensional Gaussian distribution with zero-mean and covariance $I_3$. For each value of $\sigma_1$ and $\sigma_2$, the estimation errors for 25 independent realizations are recorded. Figure~\ref{fig:orthogonal group sub-gaussian} shows the box-whisker plots of the maximum block-wise estimation error against the $\sigma_2$ (with $\sigma_1 = 1$) and $\sigma_1$ (with $\sigma_2 = 1$) on the left and right panels, respectively. The red plot corresponds to our ASE, whereas the blue plot corresponds to the two-stage approach.

\begin{figure}[htb]  
     \centering
     \includegraphics[width=1\linewidth]{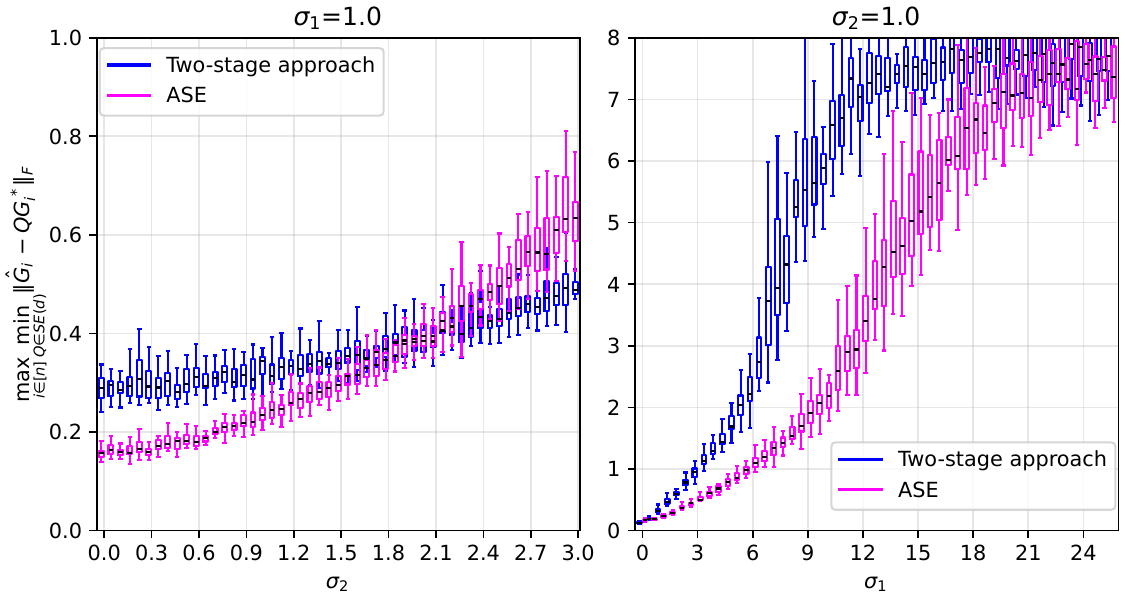}
    \caption{Estimation errors against translation and rotation noise magnitudes.}
     \label{fig:orthogonal group sub-gaussian}
\end{figure}

From the left panel of Figure~\ref{fig:orthogonal group sub-gaussian}, we see that the estimation error of ASE is majorized by that of the two-stage approach when the translation error is not too large. This shows that a holistic approach that jointly estimates the translations and rotations is preferable over the divide-and-conquer two-stage approach when the translation noise is moderate. The advantage is even more apparent from the right panel of Figure~\ref{fig:orthogonal group sub-gaussian} where we vary $\sigma_1$ with a fixed $\sigma_2 = 1$.

\subsection{Multiple Point-Set Registration}\label{sub:mpr experiment}
In the next experiment, we compare the performance of ASE against other spectral methods on the multiple point-set registration problem, which aims to optimally align a collection of 3-dimensional point clouds by transforming them using rigid motions and is naturally an instance of \SEsync. The experimental setup is similar to that in \cite{arrigoni2016spectral}. In particular, we use the Bunny, Dragon (standing), and Happy Buddha (standing) datasets from the Stanford 3D Scanning Repository \cite{stanford3dscanrep}. The noisy comparison observations $C_{ij}$ are obtained as follows. We first perturb each true relative rotation by rotating it by a random angle uniformly distributed in $[0,8^{\circ}]$ about a uniformly random axis, and each true relative translation by a zero-mean Gaussian noise with a standard deviation of 0.8 millimeters. To refine the observations, we then apply the iterative closest point algorithm~\cite{besl1992method} to the observations $C_{ij}$ as a pre-processing step. We compare our ASE with the spectral estimators developed in \cite{arrigoni2016spectral}, \cite{hadi2024se}, and \cite{doherty2022performance}. The performance metrics are the average rotation error ({\it i.e.}, the average angles (in degrees) between the estimated rotations and the corresponding true rotations (up to a common rotation)) and the average translation error ({\it i.e.}, the average distance between the estimated translations and their corresponding true translations). The experiment results are shown in Table~\ref{tab:estimation_error}.
We should point out that the approach by \cite{doherty2022performance} is unsatisfactory if we implement it as described in the paper~\cite{doherty2022performance}. To improve its practical performance, we adopt a trick from \cite{liu2023resync} to suitably flip the signs of the eigenvectors.

\begin{table}[htb]
    \centering
    \setlength{\tabcolsep}{4pt} 
    \begin{tabular}{@{}lcccc@{}}
    \hline
    Average Error & Method & Bunny & Dragon  & Happy Buddha  \\
    \hline
    \multirow{3}{*}{Rotation} 
        & ASE & \textbf{0.76} & \textbf{1.62} & \textbf{1.25} \\
        & \cite{arrigoni2016spectral} & 1.00 & 2.01 & 1.50 \\
        & \cite{hadi2024se} & 0.79 & 1.71 & 1.36 \\
        & \cite{doherty2022performance} & 1.08 & 14.28 & 1.32 \\
    \hline
    \multirow{3}{*}{Translation} 
        & ASE & \textbf{2.59} & \textbf{3.82} & \textbf{1.58} \\
        & \cite{arrigoni2016spectral} & 3.97 & 3.91 & 1.62 \\
        & \cite{hadi2024se} & 8.14 & 5.42 & 2.23 \\
        & \cite{doherty2022performance} & 98.8 & 4.72 & 2.38 \\
    \hline
    \end{tabular}
    \caption{Average rotation error (in degrees) and average translation error (in millimeters) of different spectral approaches for multiple point-set registration.}
    \label{tab:estimation_error}
\end{table}

From Table~\ref{tab:estimation_error}, see that our proposed ASE achieves the smallest estimation error with respect to both rotation and translation. The 3D models constructed using our proposed ASE are shown in Figure~\ref{fig:three_images}.

\begin{figure}[htb]  
    \centering
    \begin{subfigure}[b]{0.12\textwidth}  
        \centering
        \includegraphics[width=\linewidth]{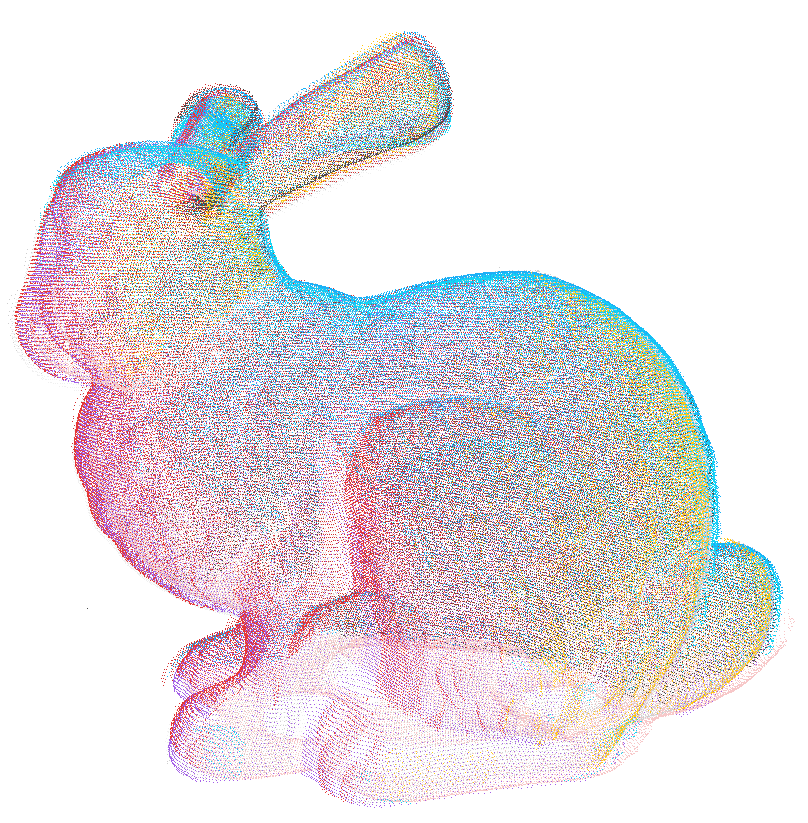} 
        \caption{Bunny}
        \label{fig:sub1}
    \end{subfigure}\hfill
    \begin{subfigure}[b]{0.15\textwidth}
        \centering
        \includegraphics[width=\linewidth]{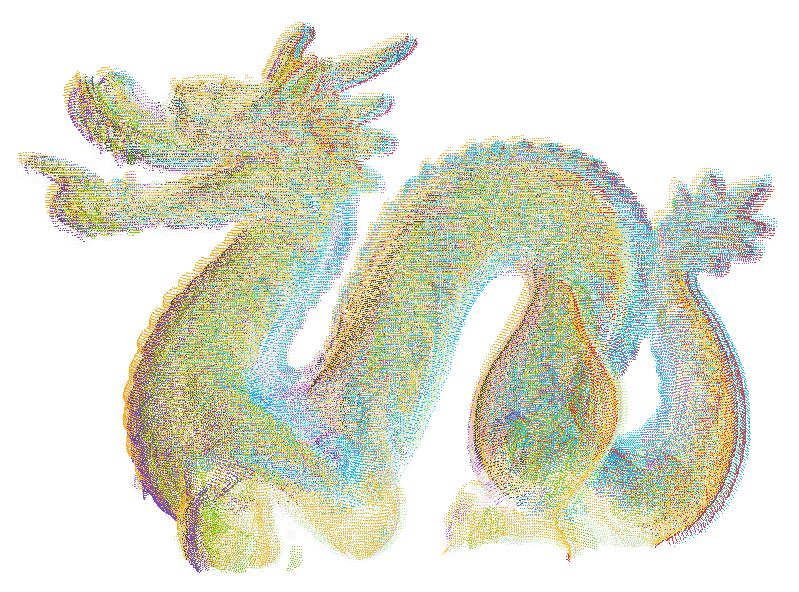}
        \caption{Dragon}
        \label{fig:sub2}
    \end{subfigure}\hfill
    \begin{subfigure}[b]{0.15\textwidth}
        \centering
        \includegraphics[width=0.7\linewidth]{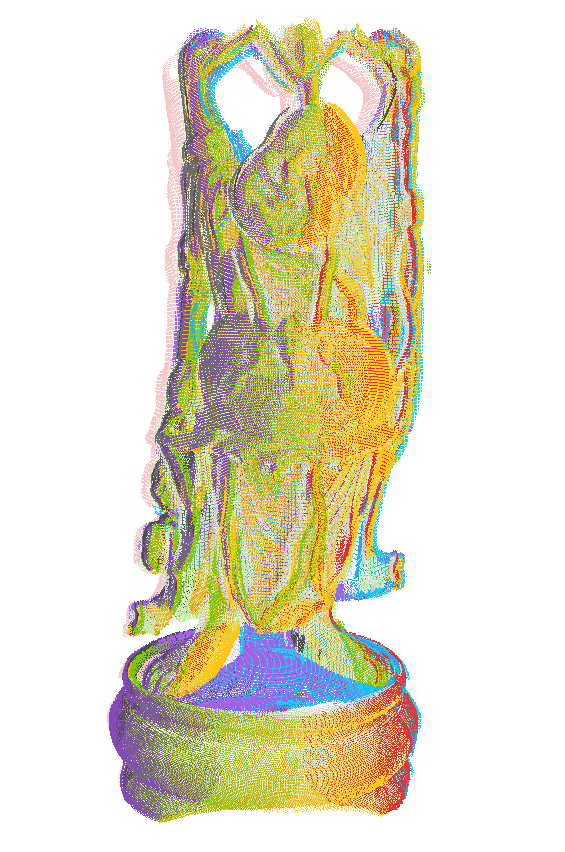}
        \caption{Happy Buddha}
        \label{fig:sub3}
    \end{subfigure}
    \caption{3D models constructed using ASE.}
    \label{fig:three_images}
\end{figure}

\section{Conclusions}
Motivated by the symmetry issue of existing spectral estimators, this paper develops a new spectral estimator, called the anchored spectral estimator (ASE), for rigid motion synchronization. ASE enjoys a strong theoretical guarantee on its estimation error. Numerically, we show that ASE outperforms a two-stage approach using synthetic data and several state-of-the-art spectral estimators through an experiment on the multiple point-set registration problem using real data.
\section{Acknowledgment}
Man-Chung Yue is supported in part by Hong Kong Research Grants Council under the GRF project 17309423.

\bibliographystyle{IEEEbib}
\bibliography{refs}
\section{Appendix}
This appendix provides the proofs of Lemmas~\ref{lem:step 1},~\ref{le: rotation error bound}, and~\ref{le: translation error bound}. 
For any matrix $H\in \mathbb{R}^{na\times b}$, we denote by $H_i \in \mathbb{R}^{a \times b}$ its $i$-th block row. In particular, for $H\in \mathbb{R}^{n\times b}$, $H_i\in\mathbb{R}^{1\times b}$ denotes its $i$-th row. For eigenvectors $\Phi$, we assume $\Phi^\top\Phi=nI_d$. Using the projection $\Pi_{\sod}$, we can map any $nd \times d$ matrix to the feasible region ${\sod}^n$ of the problem \SEsync{} using the block-wise projection $\Pi^n_{\sod} \colon \mathbb{R}^{nd \times d} \to {\sod}^n$ given by $\Pi^n_{\sod} (Y)_i = \Pi_{\sod} ([Y]_i)$, $i = 1,\ldots,n$.


\subsection{Proof of Lemma~\ref{lem:step 1}}\label{sec:inequalities for delta, phi z q}
\begin{proof}[Proof of Lemma~\ref{lem:step 1}] 
First, we define
\[
\Bar{Q}_E = \begin{pmatrix}
    \Pi_{\sod}(\Bar{Q}\Phi_1^\top) & 0 \\
  0 & 1
\end{pmatrix}\in \SEd
\]
Then, for all $i\in [n]$,
\begin{align}\label{eq: 1}
     &\min_{Q \in \SEd}\| \hat{G_i} -Q \gi^* \fnorm  \leqslant \| \hat{G_i} -\Bar{Q}_E^\top\gi^*  \fnorm \nonumber \\
     \le & \| \hat{R}_i -R_i^*  \Pi_{\sod}(\Bar{Q}\Phi_1^\top)\fnorm + \| \hat{t}_i -(\Pi_{\sod}(\Bar{Q}\Phi_1^\top))^\top t^*_i  \|_2 \nonumber \\
     =& \| \Pi_{SO(d)}(\Phi_i\Phi_1^\top)-R_i^*  \Pi_{\sod}(\Bar{Q}\Phi_1^\top)\fnorm \nonumber \\
     &\qquad +\| \hat{t}_i -(\Pi_{\sod}(\Bar{Q}\Phi_1^\top))^\top t^*_i  \|_2.
\end{align}
We then bound the first term on the RHS of \eqref{eq: 1}:
\begin{align}
 &\|  \Pi_{\sod}(\Phi_i\Phi_1^\top)-R_i^*  \Pi_{\sod}(\Bar{Q}\Phi_1^\top) \fnorm \nonumber\\
    \le &  2\| \Phi_i\Phi_1^\top-R_i^*  \Pi_{\sod}(\Bar{Q}\Phi_1^\top) \fnorm \nonumber \\
    \le & 2\| \Phi_i\Phi_1^\top-R_i^*  \Bar{Q}\Phi_1^\top \fnorm+ 2\| R_i^* \Bar{Q}\Phi_1^\top-R_i^*  \Pi_{\sod}(\Bar{Q}\Phi_1^\top) \fnorm \nonumber \\
     \le & 2 \| \Phi_1 \| \| \Phi_i-R_i^* \Bar{Q} \fnorm+ 2\| \Bar{Q}\Phi_1^\top-\Pi_{\sod}(\Bar{Q}\Phi_1^\top) \fnorm \notag \\
     \le & 2 \| \Phi_1 \| \| \Phi_i-R_i^* \Bar{Q} \fnorm + 2 \| \Bar{Q}\Phi_1^\top-R_1^{*\top} \fnorm \notag \\
     \le & 2\max_i \| \Phi_i \| \| \Phi_i-R_i^* \Bar{Q} \fnorm + 2 \| \Bar{Q}\Phi_1^\top-R_1^{*\top} \fnorm \label{eq:orthogonal-invariance},
\end{align}
where the first inequality follows from \cite[Lemma~2]{liu2023unified}, the second from the triangle inequality, and the third from the orthogonality of $R^*_i$.
Note that
\begin{equation*}
    \begin{split}
        & \sqrt{n}= \| \Phi \| = \sup_{\|x\|=1}\sqrt{\|\Phi x\|^2} = \sup_{\|x\|=1}\sqrt{(\Phi x)^\top (\Phi x)} \\
        = & \sup_{\|x\|=1}\sqrt{\sum_{i=1}^n (\Phi_i x)^2}\le \sup_{\|x\|=1}\sqrt{\sum_{i=1}^n \|\Phi_i\|^2 \|x\|^2 } \le \sqrt{n} \max_i \| \Phi_i \|.
    \end{split}
\end{equation*}
We thus have $\max_i \| \Phi_i \| \ge  1$. Combining this fact with inequalities \eqref{eq: 1}, \eqref{eq:orthogonal-invariance}, we obtain the desired result.
\end{proof}

\subsection{Proof of Lemma~\ref{le: rotation error bound}}\label{subsec: proof of lemma 3.2}
\subsubsection{Structure of the Data Matrix}\label{subsubsec:structure of data matrix}
To analyze the rotation error, we decompose the measurement matrix $\Omega$ into ground-truth and noise components. The structure of the resulting matrix is described in Lemma~\ref{le: Hi}, while the supporting estimates for $T^*$, $E$, and $EM$ (for a fixed matrix $M$) are given in Lemma~\ref{le:T}, Lemma~\ref{le:estimation of w^ti and w^tiM}, and Lemma~\ref{le: E_iM}, respectively. The constituent matrices are defined as follows: 

\begin{align*}
&T^*:= \BlkDiag(\sum_{j =1}^ns^*_{1j},\dots,\sum_{j =1}^ns^*_{nj})-s^*,\\
&\Sigma^*:=\BlkDiag(\sum_j^ns_{1j}^*s_{1j}^{*\top},\dots,\sum_j^ns_{nj}^*s_{nj}^{*\top}),\\
    & E:=\BlkDiag(\sum_{j=1}^n w^t_{1j},\dots,\sum_{j=1}^n w^t_{nj})-w^t, \\
    &\Delta :=\BlkDiag(\sum^{n}_{j\neq 1}[ s^*_{1j}w_{1j}^{t\top}+w^t_{1j}s_{1j}^{*\top}+(w^t_{1j}w_{1j}^{t\top}-\sigma_2^2I_d)],\dots,\\
    &\qquad \sum^{n}_{j\neq n}[ s^*_{nj}w_{nj}^{t\top}+w^t_{nj}s_{nj}^{*\top}+(w^t_{nj}w_{nj}^{t\top}-\sigma_2^2I_d)])\\
    & \quad -\frac{1}{2n}(ET^{*\top}+T^*E^{\top}+EE^\top)-2W^R,
\end{align*}
The matrices $T^*, E \in \mathbb{R}^{nd \times n}$ represent the ground-truth and translation noise components, respectively. Under these definitions, the measurement matrix $\Omega$ admits the following decomposition:
\begin{align*}
    \Omega&=2nI_{nd}-2R^*R^{*\top}-2W^R-\frac{1}{2n}(T^*+E)(T^*+E)^\top \\
     &+\Sigma^*+\underbrace{\BlkDiag (\sum^{n}_{j\neq i}w^t_{1j}w_{1j}^{t\top},\dots,\sum^{n}_{j\neq i}w^t_{nj}w_{nj}^{t\top})}_{\text{Quadratic term}}\\
    &+\underbrace{\BlkDiag(\sum^{n}_{\mathclap{j\neq 1}}(s^*_{1j}w_{1j}^{t\top}),\dots,\sum^{n}_{j\neq n}(s^*_{nj}w_{nj}^{t\top}))}_{\text{Cross term}}\\
        &+\underbrace{\BlkDiag(\sum^{n}_{\mathclap{j\neq 1}}(w^t_{1j}s_{1j}^{*\top}),\dots,\sum^{n}_{j\neq n}(w^t_{nj}s_{nj}^{*\top}))}_{\text{Cross term}}.
\end{align*}
The quadratic terms in $w_{ij}^t$ introduce a bias. To center the noise matrix, we define
\[
H := \Omega - \sigma_2^2(n-1)I_{nd} = 2n I_{nd} - 2 R^*R^{*\top} + \Sigma^* - \frac{1}{2n} T^*T^{*\top} + \Delta .
\]
Since $\Omega$ and $H$ share the same eigenvectors $\Phi$, analyzing the estimation error using $H$ is equivalent. For the subsequent analysis, we define the following auxiliary matrices:
\begin{align*}
&\Xi^* := \frac{n}{2}\BlkDiag(t^*_1t^{*\top}_1,\dots,t^*_nt^{*\top}_n)\\
&+\BlkDiag ( \sum_{j=1}^n  t_j^* t_j^{*\top},\dots,\sum_{j=1}^n  t_j^* t_j^{*\top} ),\\
& \Upsilon^*:= \frac{1}{2}\BlkDiag(t^*_1t^{*\top}_1,\dots,t^*_nt^{*\top}_n)(J_n\otimes I_d) \\
 &+\frac{1}{2}(J_n\otimes I_d) \BlkDiag(t^*_1t^{*\top}_1,\dots,t^*_nt^{*\top}_n)\\
    &+\frac{1}{2n}(J_n\otimes I_d) \BlkDiag(t^*_1t^{*\top}_1,\dots,t^*_nt^{*\top}_n)(J_n \otimes I_d) \\
    &-\frac{1}{2}\BlkDiag(t_1^*,\dots,t_n^*)J_n\BlkDiag(t_1^*,\dots,t_n^*)^\top.
\end{align*}
The following lemma elucidates the structure of $H$.
\begin{lemma}\label{le: Hi}
The matrix $H$ admits the decomposition
$H=2n I_{nd}-2 R^*R^{*\top}+\BlkDiag(R_1^*,\dots,R_n^*)(\Xi^*-\Upsilon^*)\BlkDiag(R_1^*,\dots,R_n^*)^\top \allowbreak +\Delta$.
\end{lemma}
\begin{proof}
 Let $BR = \BlkDiag(R_1^*,\dots,R_n^*)$. The ground-truth matrix $T^*$ can be written as
\begin{align*}
T^*:=&-BR \BlkDiag(t_1^*,\dots,t_n^*)L\\
&+BR(J_n\otimes I_d) \BlkDiag(t_1^*,\dots,t_n^*).
\end{align*}  
 Moreover,
\begin{align*}
\Sigma^* =BR[&
\BlkDiag( \sum_{j=1}^n t_j^* t_j^{*\top},\dots,\sum_{j=1}^n t_j^* t_j^{*\top})\\
&+n\BlkDiag(t_1^*t_1^{*\top},\dots,t_n^*t_n^{*\top})](BR)^\top.
\end{align*}
Consequently,
\begin{align*}
 &\Sigma^*-\frac{1}{2n}T^*T^{*\top}\\
&=BR\{
\underbrace{\BlkDiag( \sum_{j=1}^n t_j^* t_j^{*\top},\dots,\sum_{j=1}^n t_j^* t_j^{*\top})}_{\Xi^*}\\
&+\underbrace{\frac{n}{2}\BlkDiag(t_1^*t_1^{*\top},\dots,t_n^*t_n^{*\top})}_{\Xi^*}\\
&\underbrace{+\frac{1}{2}\BlkDiag(t_1^*,\dots,t_n^*)J_n\BlkDiag(t_1^*,\dots,t_n^*)^\top}_{\Upsilon^*}\\
&\underbrace{-\frac{1}{2}(J_n \otimes I_d) \BlkDiag(t_1^*t_1^{*\top},\dots,t_n^*t_n^{*\top})}_{\Upsilon^*}\\
&\underbrace{- \frac{1}{2}\BlkDiag(t_1^*t_1^{*\top},\dots,t_n^*t_n^{*\top})(J_n \otimes I_d)}_{\Upsilon^*}\\
&\underbrace{-\frac{1}{2n}(J_n \otimes I_d) \BlkDiag(t_1^*t_1^{*\top},\dots,t_n^*t_n^{*\top})(J_n \otimes I_d)}_{\Upsilon^*}\}(BR)^\top\\
&=BR(\Xi^*-\Upsilon^*)(BR)^\top.
\end{align*}
Noting that $(J_n\otimes I_d)\BlkDiag(t^*_1,\dots,t^*_n)J_n = 0$ and using the definitions of $\Xi^*$, $\Upsilon^*$ and $\Delta$, we obtain the stated decomposition of~$H$.
\end{proof}
The next lemma bounds the ground-truth translation matrix $T^*$.
\begin{lemma}\label{le:T}
 Let $T^* = T^*_D - s^*$, where $T^*_D:=\BlkDiag(\sum_js_{1j}^*,\dots,\allowbreak\sum_js_{nj}^*)$. Then, 
\begin{align*}
    &\| T^*_D \| \leqslant \Mt n ,\quad \| s^* \| \leqslant 2\Mt n,\\ &\| T^* \| \leqslant 3\Mt n ,\quad \| s^*_i \| \leqslant 2\Mt\sqrt{n}.
\end{align*}
\end{lemma}
\begin{proof}
First, $\| T^*_D \| =\| n\BlkDiag(t_1^*,\dots,t_n^*) \| \leqslant n\Mt$. For $s^*$, we have
\begin{align*}
&\| s^* \| \\
&=\| (J_n \otimes I_d)\BlkDiag(t_1^*,\dots,t_n^*)-\BlkDiag(t_1^*,\dots,t_n^*)J_n \| \\
&\leqslant \| (J_n \otimes I_d)\BlkDiag(t_1^*,\dots,t_n^*) \| + \| \BlkDiag(t_1^*,\dots,t_n^*)J_n \|\\
&\leqslant 2n\Mt.
\end{align*}
Hence, $\| T^* \| \leqslant \| T^*_D \| + \| s^* \| \leqslant 3n\Mt$. At the same time, $\| s^*_i \| \leqslant \| (t_1^*,\dots,t_n^*) \| + \| (t_1^*,\dots ,t_1^*)\|\leqslant 2\sqrt{n}\Mt$.
\end{proof}
The next lemma provides high-probability bounds for the noise matrix $E$, which is used frequently in Section~\ref{subsubsec:Rotation error bound}.
\begin{lemma}\label{le:estimation of w^ti and w^tiM}
Under the Gaussian noise model, let $E = E_D - w^t$, where $E_D:=\BlkDiag(\sum_jw^t_{1j},\dots,\sum_jw^t_{nj})$. Then, there exists an absolute constant $c > 0$ such that
     \begin{align*}
         &\| E_{D} \| \leqslant c\sigma_2(\sqrt{nd}+\sqrt{n\log n}), \quad \| w^t \| \leqslant c\sigma_2\sqrt{nd}, \\
         &\| (w^t)_i \| \leqslant c\sigma_2\sqrt{n}, \quad \| E \| \leqslant c\sigma_2(\sqrt{nd}+\sqrt{n\log n}).
     \end{align*}
with probability at least $1 - O(n^{-2})$.
\end{lemma}
\begin{proof}
We have $(E_DE_D^\top)_{ii} = (\sum_{j\neq i} w^t_{ij})(\sum_{j\neq i} w^t_{ij})^\top = Z_i Z_i^\top$, where $Z_i := \sum_{j\neq i} w_{ij}^t$ is a Gaussian random vector with variance $(n-1)\sigma_2^2 I_d$. For fixed $i$, by \cite[Theorem~3.1.1]{vershynin2018high}, we have:
\[
\left| \frac{\|Z_i\|_2}{\sigma_2 \sqrt{n-1}} - \sqrt{d} \right| \leqslant c_1 t.
\]
with probability at least $1 - 2e^{-c_2 t^2}$ for some constant $c_2>0$. This implies
\[
\|Z_i\|_2 \leqslant \sigma_2 \sqrt{n} \left( \sqrt{d} + c_1 t \right).
\]
with probability at least $1 - 2e^{-c_2 t^2}$. Squaring both sides yields
\[
\|Z_i\|_2^2 \leqslant n\sigma_2^2 \left( \sqrt{d} + c_1 t \right)^2.
\]
Choosing $t = c_3\sqrt{\log n}$ for some $c_3>0$, we obtain
\[
\|Z_i\|_2^2 \leqslant n\sigma_2^2 \left( \sqrt{d} + c_4\sqrt{\log n} \right)^2.
\]
with probability at least $1 - O(n^{-3})$. Applying the union bound, we have $\| E_D \|\leqslant c_4\sigma_2(\sqrt{nd}+\sqrt{n\log n})$ with probability at least $1 - O(n^{-2})$. 

Here $w^t \in \mathbb{R}^{nd \times n}$ is a block matrix whose $(i,j)$-th block is vector $w_{ij}^t \in \mathbb{R}^d$ for $i,j\in[n]$ where $w_{ii}^t = 0$ and $w_{ji}^t = -w_{ij}^t$ for $i \neq j$. For $i \le j$, the random vectors $w_{ij}^t$ is independent Gaussian with mean zero and covariance $\sigma_2^2 I_d$. Decompose $w^t = (w^t)^+ + (w^t)^-$, where $(w^t)^+$ contains $w_{ij}^t$ for $i \le j$ and zero otherwise, and $(w^t)^-$ contains $w_{ij}^t$ for $i \ge j$ and zero otherwise. \cite[Theorem 4.4.3]{vershynin2018high} applies for each part $(w^t)^{+}$ and $(w^t)^{-}$ separately. By a union bound, we get $\| w^t \|\leqslant c_5\sqrt{nd}$ with probability at least $1 - O(n^{-2})$. And $(w^t)_i$ is the $d\times n$  independent Gaussian matrix. Using \cite[Corollary 5.35]{vershynin2010introduction},  $\| (w^t)_i \| \leqslant c_6\sigma_2\sqrt{n}$ with probability at least $1 - O(n^{-2})$. Finally, since $E = E_D - w^t$, setting $c = c_4+ c_5+c_6$ and applying a union bound yields the desired result.
\end{proof}


\begin{lemma}\label{le: E_iM}
Let $M_{1i} \in \mathbb{R}^{1 \times d}$ and $M_{2i} \in \mathbb{R}^{d \times d}$ for $i=1,\dots,n$. Define the concatenated matrices $M^\top_1= [ M^\top_{11}, M_{12}^\top, \dots ,M^\top_{1n} ] \in \mathbb{R}^{d \times n}$ and $M_2^\top = [ M^\top_{21}, M_{22}^\top, \dots ,M^\top_{2n} ]\in \mathbb{R}^{d \times nd}$. Assume that $M_1$ is independent of $E_i$, and $M_2$ is independent of $(E^\top)_i$. Then, there exists an absolute constant $c > 0$ such that
\begin{align}
   &\|E_i M_1\fnorm \leqslant c \sigma_2(\sqrt{d} + \sqrt{\log n})( \|M_1\fnorm+ \sqrt{n}\| M_{1i} \fnorm), \label{ineq:1}\\
   & \| (E^\top)_i M_2\|_2 \leqslant c\sigma_2(\sqrt{d} + \sqrt{\log n})( \|M_2\fnorm+ \sqrt{n}\| M_{2i} \fnorm) \label{ineq:2} , 
\end{align}
with probability at least $1 - O(n^{-3})$.
\end{lemma}

\begin{proof}
Recall that $(E_{D})_i \in \mathbb{R}^{d\times n}$ is the $i$-the block row of $E_D\in \mathbb{R}^{nd\times n}$, we get $\|E_i M_1\fnorm \leqslant \| (E_{D})_i M_1\fnorm+\| (w^t)_iM_1 \fnorm$. To bound $\| w^tM_1 \fnorm$, we consider the SVD $M_1 = U \Sigma V^\top$, where $U \in \mathbb{R}^{n \times d}$ with $U^\top U = I_d$, $\Sigma \in \mathbb{R}^{d \times d}$, and $V \in \mathbb{R}^{d \times d}$. Note that $(w^t)_i$ is a $d \times n$ Gaussian random matrix. Using $\|M_1\fnorm = \|\Sigma\fnorm$, we have:
\[
\|(w^t)_i M_1\fnorm = \|(w^t)_i U \Sigma V^\top\fnorm \leqslant  \|(w^t)_i U\| \|M_1\fnorm.
\]
Since $U^\top U = I_d$, $(w^t)_i U$ is an asymmetric $d \times d$ Gaussian random matrix. \cite[Theorem~5.4]{ling2022near} guarantees that with probability at least $1 - O(n^{-3})$,
\[
\|(w^t)_i  M_1\fnorm \leqslant \|(w^t)_i U\| \|M_1\fnorm\leqslant c\sigma_2\|M_1\fnorm(\sqrt{d} + \sqrt{\log n}).
\]
Moreover,
\[ \| (E_{D})_iM_1\fnorm \leqslant \| \sum_{j}w^t_{ij} M_{1i}\fnorm \leqslant c\sigma_2 (\sqrt{nd}+\sqrt{n\log n} ) \| M_{1i} \|_2,
\]
with probability at least $1 - O(n^{-3})$, since $w^t_{ij}$ is Gaussian random vector. Combining these bounds yields inequality \eqref{ineq:1} with probability at least $1 - O(n^{-3})$. The proof for inequality \eqref{ineq:2} follows a similar argument and is therefore omitted.
\end{proof}

\subsubsection{Rotation error bound}\label{subsubsec:Rotation error bound}
\begin{proof}[Proof of Lemma~\ref{le: rotation error bound}]
We bound $\max_i \| \Phi_i - R_i^*\bar{Q} \|_{\mathrm{F}}$ similarly to \cite{liu2023resync}:
\begin{align*}
    &2n\| \Phi_i-R_i^{*} \Bar{Q} \fnorm \leqslant (2n+\lambda_{\min}(R_i^*\Xi^*_{ii}R^{*\top}_i))\| \Phi_i-R_i^{*} \Bar{Q} \fnorm \\
    &\leqslant  \|  (2nI_d+R_i^*\Xi^*_{ii}R^{*\top}_i)\Phi_i-(2nI_d+R^{*}_i\Xi^*_{ii}R^{*\top}_i)R_i^{*} \Bar{Q}\fnorm\\
    &= \| 2n\Phi_i+ R_i^*\Xi^*_{ii}R^{*\top}_i\Phi_i+\Phi_i\Lambda-(H\Phi)_{i}- 2nR_i^{*} \Bar{Q}\\
    &\qquad -R^{*}_i\Xi^*_{ii} \Bar{Q}\fnorm\\
     &\leqslant  \| 2 R^{*\top}+\Upsilon^*_i \BlkDiag(R_1^*,\dots,R_n^*)^\top\| \|\Phi-R^{*}\Bar{Q}\fnorm \\
     &\quad +\| \Delta_{i}\Phi\fnorm+ \| \Lambda\| \| \Phi_i\fnorm\\
     & \leqslant 8(\Mt^2+1)\sqrt{d} \| \Delta \| +\frac{4}{7n} \| \Delta_i\Phi \fnorm \| \Delta \| +\| \Delta_{i}\Phi\fnorm,
    \end{align*}
where $\Lambda$ is the $d \times d$ diagonal matrix whose diagonal entries are the $d$ smallest eigenvalues of $H$, $\Xi^*_{ii}$ denotes the $i$-th $d \times d$ block of the block-diagonal matrix $\Xi^*$, the equality is obtained by $\Phi_i\Lambda = (H\Phi)_i$, the third inequality follows from Lemma~\ref{le: Hi}. To prove the last inequality, we note that by Lemmas~\ref{le: the estimation of delta of GSP for SED}-\ref{le: leave-one-out} in proved Section~\ref{subsubsec: analysis of noise matrix}, there exist constants $c_0, c_1, c_2, c_3 >0$ such that if $(c_1\sigma_1+c_2\sigma_2 \Mt+c_3\sigma_2^2)\leqslant \frac{c_0\sqrt{n}}{\sqrt{d}+\sqrt{\log(n)}}$, then with probability at least $1-O(n^{-2})$ the following inequalities hold simultaneously:
\begin{align}
    &\| \Lambda\| \leqslant \| \Delta \| , \label{eq1} \\
    &\| 2R^{*\top}+ \Upsilon^*_i \BlkDiag(R_1^*,\dots,R_n^*)^\top\| \leqslant 2\sqrt{n}(\Mt^2+1), \label{eq3} \\
    &\|\Phi_i\|_{\mathrm{F}} \leqslant \frac{8(\Mt^2+1)\sqrt{d}}{7} +\frac{4}{7 n} \| \Delta_i\Phi \|_{\mathrm{F}}, \label{eq2} \\
    &\| \Phi  - R^*\Bar{Q} \|_{\mathrm{F}} \leqslant \frac{c_0\sqrt{d}\| \Delta \|}{\sqrt{n}}, \label{eq4} \\
    &\| \Delta \| \leqslant (c_1\sigma_1+c_2\sigma_2 \Mt+c_3\sigma_2^2) (\sqrt{nd}+\sqrt{n\log n}), \label{eq5}\\
    &\| \Delta_i\Phi \fnorm \nonumber\\
     &\leqslant  (\Mt^2+1)(c_1\sigma_1+c_2\sigma_2 \Mt+c_3\sigma_2^2)  (d\sqrt{n}+\sqrt{nd \log n}), \label{eq6}\\
    &\max_i \| \Phi_i \| \leqslant \max_i \| \Phi_i \fnorm \leqslant c_4(\Mt^2+1)\sqrt{d} \label{eq7},
\end{align}
Finally, applying a union bound over $i=1,\dots,n$ for $\max_i \| \Phi_i - R_i^*\bar{Q} \|_{\mathrm{F}}$, we have:
\begin{align}
    &\max_{i\in [n]}\| \Phi_i -R^*_i \Bar{Q} \|_{\mathrm{F}} \notag \\
    &\le c_4 (\Mt^2+1) (c_1\sigma_1 + c_2 \Mt \sigma_2 + c_3\sigma_2^2) \frac{(d+\sqrt{d\log n})}{\sqrt{n}}\label{eq8}
\end{align}
Combining inequality \eqref{eq8} and inequality~\eqref{eq7} yields the desired result.
\end{proof}
\subsubsection{Analysis of Noise Matrix for $\SEd$}\label{subsubsec: analysis of noise matrix}
The following six lemmas prove inequalities~\eqref{eq1}-\eqref{eq7} used in the proof of Lemma~\ref{le: rotation error bound} in Section~\ref{subsubsec:Rotation error bound}.
\begin{lemma}[Proof of inequality \eqref{eq1}]\label{le: the estimation of delta of GSP for SED}
    If $\| \Delta \| \leqslant n/4$, then the eigenvalues of the data matrix $H$ are controlled by the noise matrix, i.e., 
    $\| \Lambda \| \leqslant \| \Delta \|$ for all $1 \leqslant i \leqslant n$, 
    where $\Lambda$ is the $d \times d$ diagonal matrix whose diagonal entries are the $d$ smallest eigenvalues of $H$.
\end{lemma}
\begin{proof}
By Weyl's theorem \cite{stewart1998perturbation},
\begin{align*}
&\max_i \lvert \lambda_{i}(\Lambda)-\lambda_{i}(2 n I_{nd} -2 R^{*}R^{*\top}+\Sigma^*-\frac{1}{2}T^*T^{*\top}) \lvert \leqslant \| \Delta \|. \end{align*}
Since the smallest $d$ eigenvalues of $2n I_{nd} - 2 R^{*}R^{*\top} + \Sigma^* - \tfrac{1}{2}T^*T^{*\top}$ are zero, we obtain
$\max_i \lvert \lambda_{i}(\Lambda)-0 \lvert \leqslant \| \Delta \| $.
\end{proof}
\begin{lemma}[Proof of inequality \eqref{eq3}]\label{le: lemma for Z}
      If $ \max_i\| t^*_i \|_2 \leqslant \Mt$, $\|  2R^{*\top}+ \Upsilon^*_i \BlkDiag (R_1^*,\dots,R_n^*)^\top\| \leqslant 2\sqrt{n}(\Mt^2+1)$.
\end{lemma}
\begin{proof}
From the definition of $\Upsilon^*_i$,
\begin{align*}
    \| \Upsilon^*_i \| &\leqslant \frac{1}{2}\max_i\| t_i^*t_i^{*\top}\| \sqrt{n}+\frac{1}{2}\|  (t^*_it^{*\top}_1 \dots t^*_it^{*\top}_n)\| \\
    &+\frac{1}{2}\| (t_1^*t_1^{*\top} \dots t_n^*t_n^{*\top})\ \|+\frac{1}{2}\sqrt{n}\max_i \| t_i^*t_i^{*\top}\|\\  
  & \leqslant  2\Mt^2\sqrt{n}.
\end{align*}
 Therefore, $\| 2 R^{*\top}+\Upsilon^*_i  \BlkDiag (R_1^*,\dots,R_n^*)^\top\| \leqslant 2 \|   R^* \|+ \| \Upsilon^*_i \| \leqslant 2\sqrt{n}(\Mt^2+1)$.
\end{proof}
\begin{lemma}[Proof of inequality \eqref{eq2}]\label{le:phi for sed}
    If $\| \Delta \| \leqslant  \frac{n}{4}$ and $ \max_i\| t^*_i \|_2 \leqslant \Mt$, then $\|\Phi_i\fnorm \leqslant \frac{8(\Mt^2+1)\sqrt{d}}{7} +\frac{4}{7 n} \| \Delta_i\Phi \fnorm $.
\end{lemma}
\begin{proof}
For $\| \Phi_i\fnorm$, we have
\begin{align*}
    &(2 n+\lambda_{\min}(R_i^*\Xi^*_{ii}R_i^{*\top}))\| \Phi_i\fnorm \\
    & \leqslant \| 2n\Phi_i+R_i^*\Xi^*_{ii}R_i^{*\top} \Phi_i-(H\Phi)_i+\Phi_i\Lambda\fnorm \\
    &\leqslant \| 2n\Phi_i+ R_i^*\Xi^*_{ii}R_i^{*\top} \Phi_i-(H\Phi)_i\fnorm+\| \Phi_i\fnorm \| \Lambda\|.
\end{align*}
Using the inequality $ \| \Lambda \| \leqslant \frac{n}{4}$ and $\lambda_{\min}(R^*_i\Xi^*_{ii}R^{*\top}_i)\allowbreak \geqslant 0$ ($\Xi^*_{ii}= \frac{n}{2}t^*_it^{*\top}_i+ \sum_{j=1}^n  t_j^* t_j^{*\top}$ are positive semi-definite), we obtain
\begin{align*}
   &\frac{7n}{4} \|\Phi_i\fnorm \\
   &\leqslant   \| 2n\Phi_i+R_i^*\Xi^*_{ii}R_i^{*\top} \Phi_i-(H\Phi)_i\fnorm \\&\leqslant  \|
(2  R^{*\top}+\Upsilon^*_i \BlkDiag (R_1^*,\dots,R_n^*)^\top)\Phi \fnorm+ \| \Delta_{ i}\Phi \fnorm.
\end{align*}
By Lemma~\ref{le: lemma for Z}, we have $\|\Phi_i\fnorm \leqslant \frac{8(\Mt^2+1)\sqrt{d}}{7} +\frac{4}{7n} \| \Delta_{i}\Phi \fnorm$.
\end{proof}

\begin{lemma}[Proof of inequality \eqref{eq4}]\label{le:phi-RQ}
$\| \Phi -R^* \Bar{Q}\fnorm \leqslant \frac{c_0\sqrt{d}\| \Delta \|}{\sqrt{n}}$.
\end{lemma}
\begin{proof}
The smallest $d$ eigenvalues of $2nI_{nd}-2R^{*}R^{*\top}+\Sigma^*-\frac{1}{2n}T^*T^{*\top}$ are $0$. Each column of $\frac{1}{\sqrt{n}}R^*$ is an eigenvector. Moreover, $\| R^* \|_{\mathrm{F}} = \sqrt{nd}$. Furthermore, each column of $\frac{1}{\sqrt{n}}\Phi$ is a normalized eigenvector of $H$. Using the variant of the Davis-Kahan theorem \cite[Theorem 5]{liu2023unified}, we obtain
$\frac{1}{\sqrt{n}}\| \Phi -R^* \Bar{Q}\fnorm \leqslant \frac{c_0\sqrt{d}\| \Delta \|}{n}
$.
\end{proof}
Lemmas~\ref{le: norm of delta_sed in gaussian noise} and \ref{le: leave-one-out} and Proposition~\ref{prop: norm of Delta_sed M} rely on a decomposition of the noise matrix $\Delta$ into three parts: $2W^R$, $\Delta_T$, and $\Delta_{\Sigma}$, where
\begin{align*}
&\Delta=2W^R+\Delta_{\Sigma}-\Delta_{T},\\
    &\Delta_T:=\frac{1}{2n}(ET^{*\top}+T^*E^{\top}+EE^{\top}),\\
&\Delta_{\Sigma}:=\BlkDiag(\sum^{n}_{j\neq 1}[ s^*_{1j}w_{1j}^{t\top}+w^t_{1j}s_{1j}^{*\top}+(w^t_{1j}w_{1j}^{t\top}-\sigma_2^2I_d)],\dots,\\
    &\qquad \sum^{n}_{j\neq n}[ s^*_{nj}w_{nj}^{t\top}+w^t_{nj}s_{nj}^{*\top}+(w^t_{nj}w_{nj}^{t\top}-\sigma_2^2I_d)])
\end{align*}
Here $2W^R$ is the rotation noise in $H$, while $\Delta_T$ and $\Delta_{\Sigma}$ constitute the translation noise.
\begin{lemma}[Proof of inequality \eqref{eq1}]\label{le: norm of delta_sed in gaussian noise}
     There exist absolute positive constants $c_1$, $c_2$, $c_3$, $c_4$ such that, if $\sigma_2\leqslant \frac{\Mt\sqrt{n}}{c_4(\sqrt{d}+\sqrt{\log n})}$, then,
    $$\| \Delta \| \leqslant  (c_1\sigma_1+c_2\Mt \sigma_2+c_3\sigma^2_2)  (\sqrt{nd}+\sqrt{n \log n}). $$
     with probability at least $1 - O(n^{-2})$.
\end{lemma}
\begin{proof}
    Note that $\| \Delta \| \leqslant \| 2W^R \| + \| \Delta_T \| + \| \Delta_{\Sigma} \|$. Since $2W^R$ is a standard Gaussian matrix, $\| 2W^R \| \leqslant c_1\sigma_1 \sqrt{nd}$ with probability at least $1-O(n^{-2})$. For $\Delta_T$, we have
    \begin{align}
  \| \Delta_T \| &=
          \|\frac{1}{2n}ET^{*\top}+\frac{1}{2n}T^*E^{\top}+\frac{1}{2n}EE^{\top}\| \nonumber \\ &\leqslant  \frac{1}{2n}\| ET^{*\top} \| + \frac{1}{2n}\| T^*E^{\top} \| + \frac{1}{2n}\| EE^{\top} \| \label{ineq 1}.
    \end{align}
    By Lemma~\ref{le:T} and Lemma~\ref{le:estimation of w^ti and w^tiM},
   \begin{align}
     &\frac{1}{2n}\| ET^{*\top} \| = \frac{1}{2n}\| T^* E^{\top} \| \nonumber \\
     &\leqslant \frac{1}{2n} \| T^* \| \| E^{\top} \|  \leqslant \frac{1}{4}c_2 \Mt \sigma_2 (\sqrt{nd}+\sqrt{n\log n}),\label{ineq 2}
   \end{align}  
   with probability at least $1-O(n^{-2})$. For the higher-order term, using the condition $\sigma_2 \leqslant \frac{\Mt\sqrt{n}}{c_4(\sqrt{d}+\sqrt{\log n})}$,
   \begin{align}
      &\frac{1}{2n}\| EE^{\top} \|  \leqslant  \frac{1}{2n}\| E\| \| E^{\top} \| \leqslant  c_2^2\sigma_2^2 (\sqrt{d}+\sqrt{\log n})^2 \nonumber\\
      &\leqslant \frac{c_2}{4} \Mt\sigma_2^2 (\sqrt{nd}+\sqrt{n\log n}),\label{ineq 3}
 \end{align}
with probability at least $1-O(n^{-2})$. For $\Delta_{\Sigma}$,
  \begin{align*}
      &\| \Delta_{\Sigma} \| \\
      &\leqslant 2\max_i\| \sum^{n}_{j\neq i}w_{ij}^ts_{ij}^{*\top}\|  + \max_i\| \sum_{j\neq i}^nw_{ij}^tw_{ij}^{t\top}-\sigma_2^2(n-1)I_{d}  \| .
  \end{align*}
 For $\| \sum^{n}_{j\neq i}w_{ij}^ts_{ij}^{*\top}\|$ where $i$ is fixed, we 
apply Lemma~\ref{le: E_iM} to get $\| \sum^{n}_{j\neq i}w_{ij}^ts_{ij}^{*\top}\|\leqslant \frac{c_2}{4}\Mt\sigma_2 \allowbreak (\sqrt{nd}+\sqrt{n\log n})$ with probability at least $1-O(n^{-3})$, since $ \| (s_{i1}^{*}, s_{i2}^{*},\dots,\allowbreak s_{in}^{*})^\top \|_2 \leqslant 2\sqrt{n-1}\Mt$.

For term $\| \sum_{j\neq i}^nw_{ij}^tw_{ij}^{t\top}-\sigma_2^2(n-1)I_{d}  \|$ with fixed $i$, let $A$ be the $(n-1) \times d$ matrix whose rows are $\sigma_2^{-1}w^{t\top}_{ij}$ (each $\sigma_2^{-1}w^t_{ij}$ is a column vector in $\mathbb{R}^d$, so $\sigma_2^{-1}w_{ij}^{t\top}$ is a row vector). Then $S:=A^\top A = \sum_{j \neq i} \sigma_2^{-2}w^t_{ij} w_{ij}^{t\top}$. The second moment matrix for each row $\sigma_2^{-1}w_{ij}^{t\top}$ is $\Sigma = \sigma_2^{-2}\mathbb{E}[w_{ij}^t w_{ij}^{t\top}] =  I_d$. The sub-gaussian norm of $\sigma_2^{-1}w^t_{ij}$ is less than an absolute constant $c$ (since $\sigma_2^{-1}w^t_{ij}$ is standard Gaussian). Applying \cite[Remark~5.40]{vershynin2010introduction}, we obtain, with probability at least $1 - 2\exp(-c t^2)$,
\[
\left\|S - (n-1)I_d\right\| \leqslant (n-1)\max(\delta, \delta^2)\sigma_2^{2},
\]
where $ \delta = c_3\left(\sqrt{\frac{d}{n-1}} + \frac{t}{\sqrt{n-1}}\right)$. Setting $t = c'\sqrt{\log n}$ and assuming $n \geqslant d$, we obtain
\[
\left\|\sum_{j\neq i}w_{ij}^tw_{ij}^{t\top} - \sigma_2^2(n-1)I_d\right\| \leqslant c_3\sigma_2^2\sqrt{n}(\sqrt{d}+\sqrt{\log n}),
\]
with probability at least $1 - O(n^{-3})$. Therefore, by a union bound,
\begin{align}\label{ineq 4}
    \| \Delta_{\Sigma} \| \leqslant (\frac{c_2}{4}\Mt \sigma_2+c_3\sigma_2^2)\sqrt{n}(\sqrt{d}+\sqrt{\log n}),
\end{align}
with probability at least $1 - O(n^{-2})$. Combining inequalities \eqref{ineq 1}-\eqref{ineq 4} yields the desired result.
\end{proof}

\begin{Proposition}\label{prop: norm of Delta_sed M} 
  Let $M \in \mathbb{R}^{nd \times d}$ be a matrix independent of $\Delta_{Ri}$, $E_i$, and $(E^\top)_i$. Assume $\sigma_2 \leqslant \frac{\Mt \sqrt{n}}{c_4(\sqrt{d} + \sqrt{\log n})}$. Then, for a fixed $i$, there exist absolute constants $c_1, c_2, c_3, c_4$ such that,
\begin{align*}
      &\| \Delta_{i}M \fnorm \\
      &\leqslant c_1\sigma_1(\| M \fnorm+\sqrt{n}\| M_i \fnorm )(\sqrt{d}+\sqrt{\log n})\\
      &\quad +(c_2\Mt \sigma_2 +c_3\sigma_2^2)(\| M \fnorm+\sqrt{n}\| M_i \fnorm )(\sqrt{d}+\sqrt{\log n}),
  \end{align*}
with probability at least $1 - O(n^{-2})$.
\end{Proposition}  
\begin{proof}
   We decompose the norm as $\| \Delta_{i} M \fnorm \leqslant 2\| W^R_{i} M \fnorm +\| \Delta_{Ti}M \fnorm+\| \Delta_{\Sigma i} M\fnorm$.  For the first term, since $W^R_i$ is a standard Gaussian block row, Lemma~5.8 of \cite{ling2022near} gives 
    \begin{align*}
   \|  W^R_{i} M \fnorm \leqslant \sqrt{d}\|  W^R_{i} M \| \leqslant c_1\sigma_1 \| M \| (d+\sqrt{d\log n}),
 \end{align*}
 with probability at least $1 - O(n^{-2})$. We now bound $\| (\Delta_T)_i M \|_{\mathrm{F}}$. By definition,
  \begin{align}\label{norm of deltaT-deltaEpsilon}
    &\| (\Delta_{T})_i M\fnorm =\frac{1}{2n}\|(ET^{*\top})_i M+(T^*E^{\top})_i M+(EE^{\top})_i M\fnorm \nonumber \\
        & \leqslant\frac{1}{2n}( \| (ET^{*\top})_i M \fnorm +\| (T^*E^{\top})_i M \fnorm +\| (EE^{\top})_i M \fnorm).
    \end{align}
    The term $\| (T^*E^{\top})_i M \|_{\mathrm{F}}$ satisfies
  \begin{align*}\label{norm of T^*E^T M}
      \frac{1}{2n}\| T^*_iE^{\top}M \fnorm &\leqslant \frac{1}{2n}\| (\sum_js^*_{ij})(E^{\top})_i M \fnorm+\frac{1}{2n}\| s^*_iEM\fnorm,  
    \end{align*}
      Using Lemma~\ref{le:T} and Lemma~\ref{le:estimation of w^ti and w^tiM},   
    \begin{align*}    
    \frac{1}{2n}\| s^*_iE^{\top}M\fnorm 
    & \leqslant \frac{1}{2n} \|s^*_i \| \| E^{\top} \| \| M\fnorm \\&\leqslant \frac{c_2}{4}\sigma_2\Mt(\sqrt{d}+\sqrt{\log n})\| M\fnorm , 
    \end{align*}
     with probability $1-O(n^{-2})$. Since $M$ is independent of $(E^{\top})_i$, Lemma~\ref{le: E_iM} yields
    \begin{align*}
    &\frac{1}{2n}\| (\sum_js^*_{ij})(E^\top)_i M \fnorm \\
    &\leqslant \frac{c_2}{4} \sigma_2 \Mt(\| M \fnorm+\sqrt{n}\| M_i \fnorm )(\sqrt{d}+\sqrt{\log n}),
     \end{align*}
    with probability at least $1-O(n^{-3})$. Applying the bound for $\| E_i M \|_{\mathrm{F}}$ from Lemma~\ref{le: E_iM} to $\| (ET^{*\top})_i M \|_{\mathrm{F}}$, we obtain 
    \begin{align*}
        &\frac{1}{2n}\| E_i T^{*\top}M \fnorm \\
        &\leqslant \frac{c_2\sigma_2  (\sqrt{d}+\sqrt{\log n })}{4n}\| T^{*\top} M \fnorm \\
        & + \frac{c_2\sigma_2  (\sqrt{d}+\sqrt{\log n })}{4\sqrt{n}} (\| (T^{*\top}_{D})_i M \fnorm+\| (s^{*\top})_i M \fnorm)\\
        &\leqslant \frac{c_2\sigma_2  (\sqrt{d}+\sqrt{\log n })}{4n}\| T^{*\top}\| \| M \fnorm \\
        & + \frac{c_2\sigma_2  (\sqrt{d}+\sqrt{\log n })}{4\sqrt{n}} (\| \sum_j s_{ij}^{*} \|_2 \|  M_i \fnorm+\| (s^{*\top})_i \|_2 \| M \fnorm)\\
        & \leqslant \frac{c_2}{4} \sigma_2 \Mt (\sqrt{d}+\sqrt{\log n }) (\| M \fnorm +\sqrt{n} \| M_i \fnorm),
   \end{align*} 
   with probability $1-O(n^{-3})$. The second inequality uses $T^* = T^*_D - s^*$, and the third uses $(T^{*\top}_{D})_iM=(\sum_j s_{ij}^{*\top}) M_i$, where $T^{*\top}_{D}$ is block diagonal matrix. For $\frac{1}{2n} \| (EE^{\top})_i M \|_{\mathrm{F}}$, using the condition $\sigma_2 \leqslant \frac{\Mt \sqrt{n}}{c_4(\sqrt{d} + \sqrt{\log n})}$,
  \begin{align*}
      \frac{1}{2n}\| (EE^{\top})_iM \fnorm &\leqslant \frac{1}{2n}\| E_i \| \| E^{\top} \| \| M \fnorm \\
      &\leqslant \frac{c_2}{4} \Mt \sigma_2 (\sqrt{d}+\sqrt{\log n})  \| M \fnorm ,
  \end{align*}
  with probability $1-O(n^{-2})$. Finally, for $\| (\Delta_{\Sigma})_i M \|_{\mathrm{F}}$,
  \begin{align*}
      \| (\Delta_{\Sigma})_iM \fnorm & \leqslant \| \Delta_{\Sigma i} \| \| M_i \fnorm \\
      &\leqslant \sigma_2 (\frac{c_2}{4}\Mt +c_3\sigma_2) (\sqrt{nd}+\sqrt{n\log n})\| M_i \fnorm,
  \end{align*}
 with probability $1-O(n^{-3})$. As a result, we obtain
  \begin{align*}
     &\| \Delta_{i}M \fnorm \\
     &\leqslant c_1\sigma_1(\| M \fnorm+\sqrt{n}\| M_i \fnorm )(\sqrt{d}+\sqrt{\log n})\\
     &+(c_2\Mt \sigma_2 +c_3\sigma_2^2)(\| M \fnorm+\sqrt{n}\| M_i \fnorm )(\sqrt{d}+\sqrt{\log n}).
  \end{align*}
 Collecting the bounds for $W^R_i$, $(\Delta_T)_i$, and $(\Delta_{\Sigma})_i$, and adjusting the constants $c_1, c_2, c_3$, we obtain the desired inequality.
\end{proof}

\begin{lemma}[Proof of inequalities \eqref{eq6} and \eqref{eq7}]\label{le: leave-one-out}
Let $\Phi^{(i)}$ be the matrix whose columns are the $d$ smallest eigenvectors of the data matrix $H^{(i)}:=2nI_{nd}-R^*R^{*\top}+\Sigma^*-\frac{1}{2n}T^*T^{*\top}+\Delta^{(i)}$, where $\Delta^{(i)}$ is obtained from $\Delta$ by removing the $i$-th block row and $i$-th block column of $2W^R$, the $i$-th block row and $i$-th column of $E$, and the $i$-th diagonal block of $\Delta_{\Sigma}$. Assume $(c_1\sigma_1 + c_2\sigma_2 \Mt + c_3\sigma_2^2) \leqslant \dfrac{c_0\sqrt{n}}{\sqrt{d} + \sqrt{\log n}}$ for some absolute positive constants $c_0, c_1, c_2, c_3, c_4$ such that
\begin{align}
   &\max_i\| \Delta_i\Phi \fnorm \nonumber \\
     &\leqslant  (\Mt^2+1)(c_1\sigma_1+c_2\sigma_2 \Mt+c_3\sigma_2^2)  (d\sqrt{n}+\sqrt{nd \log n}), \\
    &\max_i \| \Phi_i \fnorm \leqslant c_4(\Mt^2+1)\sqrt{d},
\end{align}
with probability at least $1 - O(n^{-2})$.
\end{lemma}

\begin{proof}
We employ a leave-one-out argument. Note that the quantity $\| \Delta_i\Phi \fnorm $ is invariant under the orthogonal group $\od$. Define
$$
S^{(i)}:=\arg \min_{S \in \od} \| \Phi-\Phi^{(i)}S\fnorm,
S^{(i)} = \Pi_{\od}((\Phi^{(i)})^{\top}\Phi).
$$
We decompose $\| \Delta_i\Phi  \|$ into three terms and find an upper bound for each of them:
\begin{align}
    \| \Delta_i\Phi \fnorm & \leqslant \| \Delta_i(\Phi -\Phi^{(i)}S^{(i)}+\Phi^{(i)}S^{(i)})\fnorm\nonumber\\
    & \leqslant \| \Delta_i\|  \| \Phi -\Phi^{(i)}S^{(i)}\fnorm + \| \Delta_i\Phi^{(i)}\fnorm .\label{ineq01}
\end{align}
The term $\| \Delta_i \Phi^{(i)} \|_{\mathrm{F}}$ can be bounded by using Proposition~\ref{prop: norm of Delta_sed M}. We now bound $\| \Phi - \Phi^{(i)}S^{(i)} \|_{\mathrm{F}}$ using Davis--Kahan theorem. The $d$-th smallest eigenvalue of $H$ is at most $\| \Delta \|$, and the $(d+1)$-th smallest eigenvalue of $H$ is at least $2n - \| \Delta \|$. Hence,
\begin{align*}
    \| \Phi -\Phi^{(i)}S^{(i)}\fnorm
    &\leqslant \frac{\sqrt{2} \| (\Delta-\Delta^{(i)})\Phi^{(i)} \fnorm}{2n-3\| \Delta\|} \\
    &\leqslant \frac{4\sqrt{2}\| (\Delta-\Delta^{(i)})\Phi^{(i)} \fnorm}{5n},
\end{align*}
provided $\| \Delta \| \leqslant n/4$. Let $\Phi_i^{(i)}$ be the $i$-th $d \times d$ block of $\Phi^{(i)}$. Then
\begin{align*}
    &\| (\Delta-\Delta^{(i)})\Phi^{(i)} \fnorm \\
    & \leqslant 2\| W^R_{i} \| \| \Phi_{i}^{(i)} \fnorm +2 \| W^R_{i}\Phi^{(i)} \fnorm +\| \Delta_{\Sigma i} \| \| \Phi_{i}^{(i)} \fnorm  \\
    &+\frac{1}{2n} \| E_iT^{*\top} \Phi^{(i)} \fnorm +\frac{1}{2n}\| (E^\top)_i \|_2 \| T^*  \|\| \Phi_{i}^{(i)} \fnorm \\
    &+\frac{1}{2n}\| T^* \| \| (E_{D})_i  \|\| \Phi_{i}^{(i)} \fnorm+\frac{1}{2n}\| (E_{D})_i \| \| T^*  \|\| \Phi_{i}^{(i)} \fnorm \\
    & + \frac{1}{2n}\| T^* \| \| (E^\top)_i \Phi^{(i)} \fnorm+ \frac{1}{2n}\| T^* \| \|  E_i \| \| \Phi_{i}^{(i)} \fnorm \\
    &+ \frac{1}{2n} \| E \| \| E_i \Phi^{(i)} \fnorm +\frac{1}{2n} \| E \| \| E_i \| \| \Phi^{(i)}_i \fnorm \\
    &+\frac{1}{2n}\| E \| \| (E_{D})_i  \|\| \Phi_{i}^{(i)} \fnorm+\frac{1}{2n}\| (E_{D})_i \| \| E  \|\| \Phi_{i}^{(i)} \fnorm\\
    &+ \frac{1}{2n} \| E_i E^{(i)}  \Phi^{(i)} \fnorm +\frac{1}{2n} \|  (E^\top)_i\|_2 \| E^{(i)} \| \| \Phi^{(i)}_i \fnorm\\
    & \leqslant (c_1\sigma_1+c_2\sigma_2 \Mt+c_3\sigma_2^2)(\sqrt{d}+\sqrt{\log n})\| \Phi \fnorm \\
    &+(c_1\sigma_1+c_2\sigma_2 \Mt+c_3\sigma_2^2)(\sqrt{d}+\sqrt{\log n})\sqrt{n}\| \Phi_{i}^{(i)} \fnorm),
\end{align*}
with probability at least $1 - O(n^{-2})$. Here $E^{(i)}$ is obtained from $E$ by deleting its $i$-th block row and $i$-th column. Moreover, $(E^\top)_i$ denotes the $i$-th row of $E^\top \in \mathbb{R}^{n \times nd}$. And then,
\begin{align}
    &\| \Phi^{(i)}_{i}  \fnorm - \| \Phi_{i} \fnorm \\
    & \leqslant \| \Phi_{i}-\Phi^{(i)}_{i}S^{(i)}\fnorm \leqslant \| \Phi -\Phi^{(i)}S^{(i)}\fnorm \nonumber\\
    &\leqslant (c_1\sigma_1+c_2\sigma_2 \Mt+c_3\sigma_2^2)\frac{(\sqrt{d}+\sqrt{\log n})}{\sqrt{n}}(\| \Phi_{i}^{(i)} \fnorm+\sqrt{d}) \label{ineq:phi-phi_i}.
\end{align}
Take $c_0\le 1/4$. Since $(c_1\sigma_1+c_2\sigma_2 \Mt+c_3\sigma_2^2)\frac{\sqrt{d}+\sqrt{\log n}}{\sqrt{n}}\leqslant c_0\leqslant \frac{1}{4}$ and $\max_i \| \Phi_i \|_{\mathrm{F}} \geqslant \sqrt{d}$, we obtain
\begin{align}
    &\| \Phi^{(i)}_{i}  \fnorm - \| \Phi_{i} \fnorm \leqslant\| \Phi_{i}-\Phi^{(i)}_{i}S^{(i)}\fnorm  \nonumber\\
    &\leqslant \frac{1}{4}(\| \Phi_{i}^{(i)} \fnorm+\sqrt{d}) \leqslant \frac{1}{4}(\| \Phi_{i}^{(i)} \fnorm+\max_j \| \Phi_{j} \fnorm)\label{ineq02}.
\end{align}
Hence, $\| \Phi^{(i)}_i \|_{\mathrm{F}} \leqslant 2 \max_j \| \Phi_j \|_{\mathrm{F}}$. Combining inequalities \eqref{ineq01} and \eqref{ineq02}, and assuming $\| \Delta \| \le  n/4$,
\begin{align}
    &\max_i\| \Delta_i\Phi \fnorm \nonumber \leqslant \frac{3}{4}\max_i \| \Delta_i\|  \max_i\| \Phi_{i} \fnorm + \max_i\| \Delta_i\Phi^{(i)}\fnorm \nonumber\\
    & \leqslant \frac{3n}{16}\max_i\| \Phi_{i} \fnorm + \max_i\| \Delta_i\Phi^{(i)}\fnorm .\label{ineq04}
\end{align}
Combining Lemma~\ref{le:phi for sed}, inequality \eqref{ineq04}, Proposition~\ref{prop: norm of Delta_sed M}, and the assumption $(c_1\sigma_1+c_2\sigma_2 \Mt+c_3\sigma_2^2)\frac{\sqrt{d}+\sqrt{\log n}}{\sqrt{n}}\leqslant c_0\leqslant \frac{1}{4}$,
\begin{align*}
    &\max_i\| \Phi_{i} \fnorm  \leqslant \frac{8(\Mt^2+1)\sqrt{d}}{7} +\frac{4}{7 n} \| \Delta_i\Phi \fnorm\\
    &\leqslant 2(\Mt^2+1)\sqrt{d}+\frac{3}{28}\max_i\|\Phi_i \fnorm +\frac{4}{7n}\max_i \|\Delta_i \Phi^{(i)} \fnorm \\
    &\leqslant 3(\Mt^2+1)\sqrt{d}+\frac{3}{28}\max_i\|\Phi_i \fnorm +\frac{1}{7}\max_i \| \Phi_i^{(i)} \fnorm \\
    &\leqslant 3(\Mt^2+1)\sqrt{d}+\frac{3}{28}\max_i\|\Phi_i \fnorm +\frac{2}{7}\max_i \| \Phi_i \fnorm. 
\end{align*}
Therefore, $\max_i\| \Phi_{i} \fnorm \leqslant c_4(\Mt^2+1)\sqrt{d}$.
Finally, using inequality \eqref{ineq01}, Proposition~\ref{prop: norm of Delta_sed M}, and Lemma~\ref{le: norm of delta_sed in gaussian noise},
$\| \Delta_i\Phi \fnorm  \leqslant (\Mt^2+1)(c_1\sigma_1+c_2\sigma_2 \Mt+c_3\sigma_2^2)  (d\sqrt{n}+\sqrt{nd \log n})$ with probability at least $1-O(n^{-2})$.
\end{proof}

\subsection{Proof of Lemma~\ref{le: translation error bound}}\label{subsec:translation error bound}

Before we prove Lemma \ref{le: translation error bound}, we need Lemma \ref{eq: n projection inequality}, which is a simple corollary of \cite[Lemma~2]{liu2023unified}.
\begin{lemma}\label{eq: n projection inequality}
    $\| \Pi^n_{\sod}(Y) - X \|_F \leq 2 \| Y - X \fnorm$ for $Y\in \mathbb{R}^{nd \times d}$ and $X\in \mathbb{R}^{nd \times d}$ where each block $X_i \in \sod$.
\end{lemma}
\begin{proof}
Let $B:=\Pi^n_{\sod}(Y) - X $, we have
\begin{equation*}
    \begin{split}
    &\| \Pi^n_{\sod}(Y) - X \|_F^2=\tr (B^\top B)=\sum_{i=1}^n(\tr(B_i^\top B_i))\\
    &=\sum_{i=1}^n\|B_i\fnorm^2 \leqslant \sum_{i=1}^n 4\|Y_i-X_i\fnorm^2 = 4\| Y-X\fnorm^2.
    \end{split}
\end{equation*}
where the inequality is due to \cite[Lemma~2]{liu2023unified}.
\end{proof}
Here Lemma~\ref{le:phi-RQ2} is an operator norm error bound for $\|\Phi-R^*\Bar{Q}\|$:
\begin{lemma}\label{le:phi-RQ2}
There exist absolute constants $c_1, c_2, c_3, c_4 > 0$ such that if $\sigma_2\leqslant \frac{\Mt\sqrt{n}}{c_4(\sqrt{d}+\sqrt{\log n})}$, then
    $\| \Phi -R^* \Bar{Q}\| \leqslant c_0(c_1\sigma_1+c_2\Mt \sigma_2+c_3\sigma^2_2)  (\sqrt{d}+\sqrt{\log n})$,
     with probability at least $1 - O(n^{-2})$.
\end{lemma}
\begin{proof}
The smallest $d$ eigenvalues of $2nI_{nd}-2R^{*}R^{*\top}+\Sigma^*-\frac{1}{2n}T^*T^{*\top}$ are $0$. Each column of $\frac{1}{\sqrt{n}}R^*$ is an eigenvector of $2nI_{nd}-2R^{*}R^{*\top}+\Sigma^*-\frac{1}{2n}T^*T^{*\top}$. Moreover, $\| R^* \| = \sqrt{n}$. Furthermore, each column of $\frac{1}{\sqrt{n}}\Phi$ is a normalized eigenvector of $H$. 
We next apply the Davis-Kahan theorem~\cite[Theorem 6.2]{ling2022near} and \cite[Lemma 6.3]{ling2022near}. Specifically, we flip the sign of the ground truth matrix $2nI_{nd}-2R^{*}R^{*\top}+\Sigma^*-\frac{1}{2n}T^*T^{*\top}$ and the data matrix $H$, and then analyze the leading $d$ eigenvectors of both matrices using the theorem and the Lemma. This results in
\begin{align}
    \frac{1}{\sqrt{n}}\| \Phi -R^* \Bar{Q}\|&\leqslant 2\frac{1}{\sqrt{n}}\| (I_n-\frac{1}{n}\Phi \Phi^\top)R^*\| \notag\\
    &\leqslant \frac{c_0\| \Delta \frac{1}{\sqrt{n}}R^* \|}{n} \leqslant\frac{c_0\| \Delta \|}{n}. \label{eqtt0}
\end{align}
According to the estimation in Lemma~\ref{le: norm of delta_sed in gaussian noise}, we have $\|\Delta \| \leqslant (c_1\sigma_1+c_2\Mt \sigma_2+c_3\sigma^2_2)  (\sqrt{nd}+\sqrt{n \log n})$. Substituting this into \eqref{eqtt0} completes the proof.
\end{proof}
\begin{proof}[Proof of Lemma~\ref{le: translation error bound}]
From the structure of $\hat{T}$, for any fixed $i$,
\begin{align*}
&\| \hat{t}_i - (\Pi_{\sod}(\Bar{Q}\Phi_1^\top))^\top t_i^* \|_2\\
& = \frac{1}{2n}\|  (T^{*\top})_iR^{*}\Pi_{\sod}(\Bar{Q}\Phi_1^\top)-(\hat{T}^\top)_i\Pi_{\sod}^n(\Phi \Phi_1^\top )\|_2 \\
&\leqslant \frac{1}{2n}\| nt_i^{*\top}R_i^{*\top}R_i^{*}\Pi_{\sod}(\Bar{Q}\Phi_1^\top)-(nt_i^{*\top}R_i^{*\top})\Pi^n_{\sod}(\Phi \Phi_1^\top )_i\|_2\\
&+\frac{1}{2n}\| (s^{*\top})_iR^{*}\Pi_{\sod}(\Bar{Q}\Phi_1^\top)-(s^{*\top})_i\Pi^n_{\sod}(\Phi \Phi_1^\top )\|_2 \\
&+\frac{1}{2n}\|(E_D)_i\Pi_{\sod}^n(\Phi \Phi_1^\top )\|_2+\frac{1}{2n}\| ((w^t)^{\top})_i\Pi^n_{\sod}(\Phi \Phi_1^\top )\|_2.
\end{align*}

The above inequality uses the decomposition $(\hat{T}^\top)_i=(\hat{T}^\top_{D})_i-(s^\top)_i=(T^{*\top}_{D})_i+(E_D)_i-(s^{*\top})_i+((w^t)^{\top})_i$ and the triangle inequality. Moreover, we have
\begin{align}
&\| \hat{t}_i - (\Pi_{\sod}(\Bar{Q}\Phi_1^\top))^\top t_i^* \|_2 \notag \\
    & \leqslant \Mt\underbrace{\| R_i^{*}\Pi_{\sod}(\Bar{Q}\Phi_1^\top)-\Phi_i \Phi_1^\top\fnorm}_{\textcircled{1}}+\underbrace{\frac{\sqrt{d}}{2n}\| \sum_{j\neq i}^n w_{ij}^\top \|_2}_{\textcircled{2}} \notag\\
&+\underbrace{\frac{1}{2n}\| (s^{*\top})_i \|_2 }_{\textcircled{3}}\underbrace{\| R^{*}\Pi_{\sod}(\Bar{Q}\Phi_1^\top)-\Phi \Phi_1^\top\fnorm}_{\textcircled{4}} \notag\\
&+\underbrace{\frac{1}{2n}\| ((w^t)^{\top})_i\Pi^n_{\sod}(\Phi \Phi_1^\top )\|_2 }_{\textcircled{5}}, \label{eqt0}
\end{align}
where we use the fact $\Mt = \max_{i\in[n]}\| t_i^* \|_2$, $R^*_i$ and $\Pi_{\sod}(\Phi \Phi_1^\top )_i$ orthogonal matrices, and Lemma \ref{eq: n projection inequality} to obtain term \textcircled{4}. Therefore, we bound each term in \eqref{eqt0} separately. For term \textcircled{1} in \eqref{eqt0},
\begin{align}
    &\textcircled{1}\leqslant \| R_i^{*}\Pi_{\sod}(\Bar{Q}\Phi_1^\top)-R_i^{*}\Bar{Q}\Phi_1^\top\fnorm+\| R_i^{*}\Bar{Q}\Phi_1^\top-\Phi_i \Phi_1^\top\fnorm \notag\\
    & \leqslant \| \Pi_{\sod}(\Bar{Q}\Phi_1^\top)-\Bar{Q}\Phi_1^\top\fnorm+\| R_i^{*}\Bar{Q}-\Phi_i\fnorm \| \Phi_1\| \notag \\
    & \leqslant 2\| \Phi_1- R_1^{*} \Bar{Q}\fnorm +\| \Phi_i- R_i^{*}\Bar{Q} \fnorm \| \Phi_1\| \notag \\
    &\leqslant c_4(\Mt^2+1)^2(c_1\sigma_1+c_2\sigma_2 \Mt+c_3\sigma_2^2)\frac{(d\sqrt{d}+d\sqrt{\log n})}{\sqrt{n}}, \label{eqt1}
\end{align}
with probability at least $1-O(n^{-2})$. The third inequality in \eqref{eqt1} follows from \cite[Lemma~2]{liu2023unified}, as in the proof of Lemma~\ref{lem:step 1}. The last inequality in \eqref{eqt1} uses \eqref{eq8} from the proof of Lemma~\ref{le: rotation error bound}, as we have the same conditions of Lemma~\ref{le: rotation error bound}, and $\max_i\| \Phi_i \| \leqslant \max_i\| \Phi_i \fnorm = O((\Mt^2+1)\sqrt{d})$ in Lemma~\ref{le: leave-one-out}.

\begin{align}
    &\textcircled{2}=\frac{\sqrt{d}}{2n}\| \sum_{j\neq i}^n w_{ij}^\top\|_2 \leqslant c_5\sigma_2\frac{(d+\sqrt{d\log n})}{\sqrt{n}}, \label{eqt2}
\end{align}
with probability at least $1-O(n^{-2})$ by simply using Lemma \ref{le:estimation of w^ti and w^tiM}. And we also have
 \begin{align}
    &\textcircled{3}=\frac{1}{n}\| (s^{*\top})_i \|_2 = \frac{1}{n} \sqrt{ \sum_js^{*\top}_{ij}s^{*}_{ij}}\leqslant \frac{c_6\Mt}{\sqrt{n}} . \label{eqt3}
 \end{align}
The quantity~\textcircled{4} can be bound similarly to the term~\textcircled{1}:
\begin{align}
      &\textcircled{4}\leqslant \| R^{*}\Pi_{\sod}(\Bar{Q}\Phi_1^\top)-R^{*}\Bar{Q}\Phi_1^\top\fnorm+\| R^{*}\Bar{Q}\Phi_1^\top-\Phi \Phi_1^\top\fnorm \notag \\
    & \leqslant \sqrt{n}\| \Pi_{\sod}(\Bar{Q}\Phi_1^\top)-\Bar{Q}\Phi_1^\top\fnorm+\| R^{*}\Bar{Q}-\Phi \fnorm \| \Phi_1\| \notag \\
    & \leqslant 2\sqrt{n}\| \Phi_1- R_1^{*}\Bar{Q} \fnorm + \| \Phi- R^{*} \Bar{Q}\fnorm \| \Phi_1\|  \notag \\
    &\leqslant (\Mt^2+1)^2(c_1\sigma_1+c_2\sigma_2 \Mt+c_3\sigma_2^2)(d\sqrt{d}+d\sqrt{\log n}), \label{eqt4}
\end{align}
with probability at least $1-O(n^{-2})$. The last inequality in \eqref{eqt4} follows from Lemma~\ref{le:phi-RQ} and inequality \eqref{eq8} in Lemma~\ref{le: rotation error bound}, and $\max_i\| \Phi_i \| \leqslant \max_i\| \Phi_i \fnorm = O((\Mt^2+1)\sqrt{d})$ in Lemma~\ref{le: leave-one-out}.

For term \textcircled{5}, note that $((w^t)^\top)_i$ appears in $\Delta_i$. 
Similarly to bounding the rotation error, we apply a leave-one-out argument and borrow definitions $\Phi^{(i)}$ and $S^{(i)}$ in Lemma~\ref{le: leave-one-out}.
First, we rewrite \textcircled{5} as follows:
\begin{align*}
     \textcircled{5}= \frac{1}{2n}\| ((w^t)^\top)_i[\Pi_{\sod}^n(\Phi \Phi_1^\top )&-\Phi\Phi_1^\top+\Phi\Phi_1^\top
     \\&-\Phi^{(i)}S^{(i)}\Phi_1^\top+\Phi^{(i)}S^{(i)}\Phi_1^\top] \|_2.
\end{align*}
By the triangle inequality, we obtain
\begin{align}
     &\textcircled{5}\leqslant \frac{1}{2n}\| ((w^t)^\top)_i \|_2 \| \Pi_{\sod}^n(\Phi \Phi_1^\top )-\Phi\Phi_1^\top\| \notag
     \\&+\frac{1}{2n}\| ((w^t)^\top)_i \|_2\|\Phi\Phi_1^\top-\Phi^{(i)}S^{(i)}\Phi_1^\top\fnorm \notag
     \\
     &+\frac{1}{2n}\|((w^t)^\top)_i \Phi^{(i)}S^{(i)}\Phi_1^\top \|_2 .\label{eqt5}
\end{align}
By Lemma~\ref{le: leave-one-out}, we get $\max_i\| \Phi_i \| = O((\Mt^2+1)\sqrt{d})$, which together with \eqref{eqt5}, yields
     \begin{align}
    &\textcircled{5}\leqslant \frac{1}{2n}\| ((w^t)^\top)_i \|_2 \| \Pi_{\sod}^n(\Phi \Phi_1^\top )-\Phi\Phi_1^\top\| \notag\\
    &+c_8\frac{(\Mt^2+1)\sqrt{d}}{n}\| ((w^t)^\top)_i \|_2 \| \Phi-\Phi^{(i)}S^{(i)} \fnorm \notag\\
    &+c_8\frac{(\Mt^2+1)\sqrt{d}}{n}\| ((w^t)^\top)_i \Phi^{(i)} \|_2 . \label{eqt6}
\end{align}
Then, the first term in inequality \eqref{eqt6} can be bounded as follows:
\begin{align}
\frac{1}{2n}\| ((w^t)^\top)_i \|_2 & \| \Pi_{\sod}^n(\Phi \Phi_1^\top )-\Phi\Phi_1^\top\| \notag \\
&\leqslant c_9\frac{\sqrt{d}}{\sqrt{n}} \| \Pi_{\sod}^n(\Phi \Phi_1^\top )-\Phi\Phi_1^\top\| \notag \\
&\leqslant  c_9\frac{\sqrt{d}}{\sqrt{n}}  \|R^*\Pi_{\sod}( \Bar{Q}\Phi_1^\top )-\Phi\Phi_1^\top\|, \label{eqtt2}
\end{align}
where we use $\| ((w^t)^\top)_i \|_2\leqslant c\sigma_2\sqrt{nd}$ (due to Lemma~\ref{le:estimation of w^ti and w^tiM}). The last term in~\eqref{eqtt2} can be bounded as follows:
\begin{align}
&\frac{\sqrt{d}}{\sqrt{n}}  \|R^*\Pi_{\sod}( \Bar{Q}\Phi_1^\top )-\Phi\Phi_1^\top\|\notag\\
&\leqslant \frac{\sqrt{d}}{\sqrt{n}}\| R^{*}\Pi_{\sod}(\Bar{Q}\Phi_1^\top)-R^{*}\Bar{Q}\Phi_1^\top\|+\frac{\sqrt{d}}{\sqrt{n}}\| R^{*}\Bar{Q}\Phi_1^\top-\Phi \Phi_1^\top\| \notag \\
& \leqslant \sqrt{d}\| \Pi_{\sod}(\Bar{Q}\Phi_1^\top)-\Bar{Q}\Phi_1^\top\|+\frac{\sqrt{d}}{\sqrt{n}}\| R^{*}\Bar{Q}-\Phi \| \| \Phi_1\| \notag \\
& \leqslant \sqrt{d}\| \Phi_1- R_1^{*}\Bar{Q} \| +\frac{\sqrt{d}}{\sqrt{n}} \| \Phi- R^{*} \Bar{Q}\| \| \Phi_1\|  \notag \\
& \leqslant  (\Mt^2+1)(c_1\sigma_1+c_2\sigma_2 \Mt+c_3\sigma_2^2)\frac{(d\sqrt{d}+d\sqrt{\log n}) }{\sqrt{n}}, \label{eqt7}
\end{align}
where the last inequality in \eqref{eqt7} follows from Lemma~\ref{le:phi-RQ2},  inequality \eqref{eq8}, and the fact $\max_i\| \Phi_i \| = O((\Mt^2+1)\sqrt{d})$.

Meanwhile, using \eqref{ineq:phi-phi_i}, Lemma~\ref{le:estimation of w^ti and w^tiM}, Lemma~\ref{le: leave-one-out}, and the assumption $(c_1\sigma_1+c_2\sigma_2 \Mt+c_3\sigma_2^2) \leqslant \frac{c_0\sqrt{n}}{\sqrt{d}+\sqrt{\log n}}$, we have $\| \Phi -\Phi^{(i)}S^{(i)}\fnorm \leqslant (\Mt^2+1)(c_1\sigma_1+c_2\sigma_2 \Mt+c_3\sigma_2^2)\frac{(d+\sqrt{d\log n})}{\sqrt{n}}$ with probability at least $1-O(n^{-2})$. Substituting this into \eqref{eqt6} yields
\begin{align}
    &\frac{(\Mt^2+1)\sqrt{d}}{n}\| ((w^t)^\top)_i \|_2 \| \Phi-\Phi^{(i)}S^{(i)} \fnorm  \notag \\
    &\leqslant c_{10}(\Mt^2+1)^2(c_1\sigma_1+c_2\sigma_2 \Mt+c_3\sigma_2^2)\frac{(d\sqrt{d}+d\sqrt{\log n})}{\sqrt{n}}, \label{eqt8}
\end{align}
where we used the fact that $\| ((w^t)^\top)_i \|_2\leqslant c\sigma_2\sqrt{nd} \leqslant \frac{n}{4}$ (due to Lemma~\ref{le:estimation of w^ti and w^tiM} and the assumption that $\sigma_2 \leqslant c'\frac{\sqrt{n}}{\sqrt{d}}$ for some constant $c$). 

Recall the definition of $\Phi^{(i)}$ in Lemma~\ref{le: leave-one-out}. It is independent of of the $i$-th block row and $i$-th column of $E:=\BlkDiag(\sum_{j=1}^n w^t_{1j},\allowbreak \dots,\sum_{j=1}^n w^t_{nj})-w^t$. Therefore, $((w^t)^\top)_i$ in the third term in \eqref{eqt6} is independent of $\Phi^{(i)}$. Furthermore, following the same argument for bounding the third term in~\eqref{eqt6}, we obtain
\begin{align}
     \frac{(\Mt^2+1)\sqrt{d}}{2n}&\| ((w^t)^\top)_i\Phi^{(i)} \|_2 \notag \\
     &\leqslant c_{11}\sigma_2 (\Mt^2+1)^2\frac{(d\sqrt{d}+d\sqrt{\log n})}{\sqrt{n}}, \label{eqt9}
\end{align}
with probability at least $1-O(n^{-2})$.
Finally, substituting~\eqref{eqt7}, \eqref{eqt8} and \eqref{eqt9} into \eqref{eqt6}, we obtain
\begin{align}
     &\textcircled{5}\leqslant  c_{12}(\Mt^2+1)^2(c_1\sigma_1+c_2\sigma_2 \Mt+c_3\sigma_2^2)\frac{(d\sqrt{d}+d\sqrt{\log n}) }{\sqrt{n}}.  \label{eqt10} 
\end{align}
Under the same assumption on the noise magnitudes and using the bounds for \textcircled{1}-\textcircled{5}, i.e., \eqref{eqt1}, \eqref{eqt2}, \eqref{eqt3}, \eqref{eqt4} and \eqref{eqt10} into \eqref{eqt0}, for a fixed $i$, we have
\begin{align*}
   &\| \hat{t}_i - (\Pi_{\sod}(\Bar{Q}\Phi_1^\top))^\top t_i^* \|_2 \\
   &\leqslant c_{13}\Mt(\Mt^2+1)^2(c_1\sigma_1+c_2\sigma_2 \Mt+c_3\sigma_2^2)\frac{(d\sqrt{d}+d\sqrt{\log n})}{\sqrt{n}},
\end{align*}
with probability at least $1 - O(n^{-2})$. An application of the union bound over $i = 1,\dots,n$ completes the proof. 
\end{proof}



\end{document}